\documentclass[11pt]{amsart}

\usepackage{color}
\usepackage[pdftex]{graphicx}
\usepackage{todonotes}
\usepackage{indentfirst,pict2e}

\usepackage{amsfonts,amssymb}

\usepackage[leqno]{amsmath}
\usepackage{amsthm}
\usepackage{latexsym}

\def\docpdftitle{sln level 1 conformal blocks divisors on M0,n}
\usepackage[
	hyperindex,
	pagebackref,
	pdftex,
	pdftitle={\docpdftitle},
	pdfauthor={Maxim Arap, Angela Gibney, James Stankewicz, and David Swinarski},
	pdfdisplaydoctitle,
	pdfpagemode=UseNone,
	breaklinks=true,
	extension=pdf,
	bookmarks=false,
	plainpages=false,
	colorlinks,
	linkcolor=linkblue,
	citecolor=linkblue,
	urlcolor=linkblue,
	pdfmenubar=true,
	pdftoolbar=true,
	pdfpagelabels,
	pdfpagelayout=SinglePage,
	pdfview=Fit,
	pdfstartview=Fit
]{hyperref}

\hypersetup{pdftitle={sl\_n level 1 conformal blocks divisors on M\_0,n}}

	\definecolor{linkred}{rgb}{0.7,0.2,0.2}
	\definecolor{linkblue}{rgb}{0,0.2,0.6}

\newcommand{\myneturltilde}[2]{\href{#1}{\color{linkblue}{#2}}}

\newcommand{\neturl}[1]{\href{#1}{{\sffamily{\texttt{#1}}}}}
\newcommand{\neturltilde}[2]{\href{#1}{{\sffamily{\texttt{#2}}}}}

\usepackage[backrefs,msc-links,nobysame]{amsrefs}

\makeatletter

\def\bib@div@mark#1{%
 \@mkboth{{#1}}{{#1}}%
	}
\def\print@backrefs#1{%
 \space\SentenceSpace$\leftarrow$\csname br@#1\endcsname
}
\catcode`\'=11 
\renewcommand{\PrintAuthors}[1]{%
 \ifx\previous@primary\current@primary
  \sameauthors\@empty
 \else
  \def\current@bibfield{\bib'author}%
		  \PrintNames{}{}{\scshape #1}%
 \fi
}
\catcode`\'=12
\def\MRhref#1#2{%
 \begingroup
  \parse@MR#1 ()\@empty\@nil%
  \href{\MR@url}{\texttt{\@tempd\vphantom{()}}}%
  \ifx\@tempe\@empty
  \else
   \ \href{\MR@url}{\texttt{(\@tempe)}}%
  \fi
 \endgroup
}%
\def\MR#1{%
 \relax\ifhmode\unskip\spacefactor3000 \space\fi
 \begingroup
  \strip@MRprefix#1\@nil
  \edef\@tempa{\@nx\MRhref{MR\@tempa}{\@tempa}}%
 \@xp\endgroup
 \@tempa
}
\makeatother

\numberwithin{equation}{section}

\newtheoremstyle{claim}
{}
{}
{\itshape}
{}
{\bfseries}
{}
{1em}
{}
\theoremstyle{claim}
\newcounter{letteredtheoremcounter}

\newtheorem{theorem}[equation]{Theorem}
\newtheorem{letteredtheorem}[letteredtheoremcounter]{Theorem}
\newtheorem{theorem*}{Theorem}
\newtheorem{definition}[equation]{Definition}

\newtheorem{lemma}[equation]{Lemma}
\newtheorem{corollary}[equation]{Corollary}

\newtheorem{proposition}[equation]{Proposition}

\newtheoremstyle{explain}
{}
{}
{}
{}
{\scshape}
{}
{.5em}
{}
\theoremstyle{explain}
\newtheorem{example}[equation]{Example}
\newtheorem{remark}[equation]{Remark}

\newcommand{\C}{\mathbb{C}}
\newcommand{\R}{\mathbb{R}}
\newcommand{\Z}{\mathbb{Z}}
\newcommand{\N}{\mathbb{N}}
\newcommand{\Q}{\mathbb{Q}}

\newcommand{\Pro}{\mathbb{P}}

\newcommand{\w}{{\mathbf{w}}}
\newcommand{\dblq}{{/\!/}}

\newcommand{\M}{\overline{\mathcal{M}}} 
\newcommand{\Mgn}{\overline{\mathcal{M}}_{g,n}}
\newcommand{\F}{F}

\newcommand{\xout}[1]{\widehat{#1}}

\newcommand{\Nef}{\operatorname{Nef}}

\newcommand{\SymNef}{\operatorname{SymNef}}
\newcommand{\Ghat}{\widehat{\mathcal{G}}}
\newcommand{\familyG}{\mathcal{G}}
\newcommand{\NS}{\operatorname{NS}}
\newcommand{\Pic}{\operatorname{Pic}}

\newcommand{\sL}{\mathfrak{sl}}
\newcommand{\SL}{\operatorname{SL}}
\newcommand{\VV}[3]{\mathbb{V}(#1;#2;#3)}

\begin{document}

\pagenumbering{arabic}

\title{$\sL_n$ level 1 conformal blocks divisors on $\M_{0,n}$}

\author{Maxim Arap}
\email{arapmv@math.uga.edu}

\author{Angela Gibney}
\email{agibney@math.uga.edu}

\author{James Stankewicz}
\email{jstankew@math.uga.edu}

\author{David Swinarski}
\email{davids@math.uga.edu}

\subjclass[2000]{Primary 14D21, 14E30 \\Secondary 14D22, 81T40}
\keywords{moduli space, vector bundles, conformal field theory}

\date{\today}

\begin{abstract}We study a family of  semiample
 divisors on the moduli space $\M_{0,n}$ that come from the theory of
conformal blocks for the Lie algebra $\sL_n$ and level 1.  The divisors we study are invariant under the action of $S_n$ on $\M_{0,n}$.  We compute their classes and
prove that they generate extremal rays in the cone of symmetric nef
divisors on $\M_{0,n}$. In particular,
these divisors define birational contractions of $\M_{0,n}$, which we show
factor through reduction morphisms to moduli spaces of weighted
pointed curves defined by Hassett.
\end{abstract}

\maketitle

\section{Introduction}
The moduli spaces $\M_{g,n}$ of Deligne-Mumford pointed stable curves of genus $g$ are central objects of study in algebraic geometry and mathematical physics.  
Given a simple Lie algebra $\mathfrak{g}$, a positive integer $\ell$ called the {\em{level}},  and an $n$-tuple of dominant integral weights $\w = (\lambda_{1},\ldots,\lambda_{n})$ with each $\lambda_{i}$ of level $ \leq \ell$, the WZW model of conformal field theory can be interpreted as defining vector bundles on $\M_{g,n}$ whose fibers are the so-called vector spaces of conformal blocks.  These vector bundles were first constructed by Tsuchiya, Ueno, and Yamada \cite{Ueno}; their ranks are computed by the famous Verlinde formulae \cite{Verlinde}, \cite{Faltings}.

 In this work, we study the divisors $D^n_{1,\w}$ on $\M_{0,n}$ associated to the conformal blocks bundles given by $\sL_{n}$, level 1, and weights $\w = (\omega_{j_1},\ldots,\omega_{j_n})$.  Here the $\omega_i$ denote the fundamental dominant weights.  Fakhruddin showed that for $\sL_n$ level 1, conformal blocks vector bundles have rank 1 \cite{Fakh}*{5.2.5}.   If the set of weights $\w=(\omega_{j},\ldots, \omega_j)$ is $S_n$-symmetric, then we write $D^n_{1,j}$ for the corresponding divisor.

Conformal blocks bundles on $\M_{0,n}$ are globally generated (\cite{Fakh}*{Lemma 2.2}) and hence they define morphisms to projective varieties $\M_{0,n} \rightarrow X$.  The morphisms determined by the conformal blocks bundles given by  $\sL_{2}$, level $\ell$, and weights $\w$ were shown by Fakhruddin to factor through the contraction maps from $\M_{0,n}$ to the weighted pointed moduli spaces defined by Hassett (\cite{Fakh}*{Prop. 4.7}).  We prove an analogous statement for conformal blocks bundles given by  $\sL_{n}$, level $1$, and weights $\w$.   Namely, in Lemma \ref{slnHassettSpace} we identify what we call the \emph{minimal Hassett space $\M_{0,\mathcal{A}}$ determined by $D^n_{1,\w}$}, and we prove the following result about the associated contractions (see Section \ref{slnmorphisms}).

 \begin{letteredtheorem} \label{thmA} Let $\M_{0,\mathcal{A}}$ be the minimal Hassett space determined by $D^n_{1,\w}$.  Then the morphism $\M_{0,n} \rightarrow X$ associated to the divisor $D^n_{1,\w}$ factors through the natural birational contraction $\rho_{\mathcal{A},N}: \M_{0,n} \rightarrow \M_{0,\mathcal{A}}$ defined by Hassett.     
\end{letteredtheorem}

In Proposition \ref{nclass} we give the classes of the divisors $D^n_{1,j}$, and we learn more about their associated morphisms by finding their position in the cone of nef divisors on $\M_{0,n}$.  Recall that a divisor $D$ on a proper variety $X$ is \emph{nef} if $D$ nonnegatively intersects every curve on $X$. The set of classes of nef divisors on $X$ generates a closed cone $\Nef(X)$ in the N\'eron-Severi  space $\NS(X)_{\R}$. 
  (Note that on $\M_{0,n}$, algebraic and rational equivalence are the same, so $\Pic(\M_{0,n})_{\Q} \cong \NS(\M_{0,n})_{\Q}$, \cite{KeelIntersection}.)   The divisors $D^{n}_{1,j}$  are $S_n$-invariant, and are elements of  the cone of symmetric nef divisors, $\SymNef(\M_{0,n}) = \Nef(\M_{0,n}) \cap \Pic(\M_{0,n})_{\R}^{S_n}$.  It is known that $S_n$-invariant nef divisors on $\M_{0,n}$ are big (cf. \cite{KeelMcKernanContractible}*{Thm. 1.3(2)}, \cite{GibneyNumerical}*{Prop. 4.5}), and so the associated morphisms are birational.     

\begin{letteredtheorem}\label{extremal} 
The divisors $D^n_{1,j}$ span extremal rays of $\SymNef(\M_{0,n})$ for $j=2$, $\ldots$, $\lfloor \frac{n}{2}\rfloor$.
\end{letteredtheorem}

To prove Theorem \ref{extremal}, for each divisor $D^{n}_{1,j}$ we define a
family of rational curves $\mathcal{G} = \mathcal{G}^{n}_{1,j}$
(cf. Definitions \ref{k geq 3 family} and \ref{k=2 family}).    Given $j \in \{2, \ldots, \lfloor \frac{n}{2} \rfloor\}$, the family $\mathcal{G}^{n}_{1,j}$ consists of $\lfloor \frac{n}{2} \rfloor -1$  independent effective curves (thought of as classes in $H_2(\M_{0,n},\Q)$)  each of which  intersects  the divisor $D^{n}_{1,j}$ in degree zero.  In fact, the curves in $\mathcal{G}^{n}_{1,j}$ are \emph{$\F$-curves} (cf. \ref{curvessection}), which are special, combinatorially defined rational curves in the moduli space $\M_{0,n}$.  

In order to show that the curves in $\mathcal{G} =
\mathcal{G}^{n}_{1,j}$ are independent, which is the hardest part of
this paper, we use a convenient basis of $\F$-curves (cf. Proposition
\ref{C1}).   In Proposition \ref{gamma j},  we give a formula for
expressing any $\F$-curve in that basis.  This formula may be a useful
tool for other purposes. 

We now say a few words about how we came to study the family of divisors $\{ D^{n}_{1,j}\}$.  

Fakhruddin has recently found recursive formulas for the classes of general conformal blocks divisors on $\M_{0,n}$ \cite{Fakh}.  Swinarski implemented these formulas in a free and open source  \texttt{Macaulay 2} \cite{Macaulay} package called \texttt{ConfBlocks} \cite{ConfBlocks}, which 
can compute classes of conformal blocks divisors for many different Lie algebras and levels.  After some experimentation using the programs \texttt{ConfBlocks}, \texttt{NefWiz} \cite{NefWiz}, and \texttt{polymake} \cite{Polymake}, our group noticed that the divisors $D^{n}_{1,j}$ often spanned extremal rays of the nef cone.  

We also observed that many of the divisors $D^{n}_{1,j}$ lie in a part of the nef cone that was not previously well understood.   More precisely: the Ray Theorem of Keel and McKernan \cite{KeelMcKernanContractible} and its extension by Farkas and Gibney \cite{FarkasGibney} gives good understanding of the cone in $\NS(\M_{0,n})_{\R}$ generated by log canonical divisors.  If a divisor $D$ is log canonical and intersects every $\F$-curve nonnegatively, then $D$ is nef, and by \cite{BCHM}, semiample.   However, many of the divisors $D^{n}_{1,j} $ are not log canonical, and hence proving that they are nef and semiample requires different methods.    

We were also able to identify the images of the morphisms associated to some of the divisors $D^{n}_{1,j}$.  For instance, if $n$ is even, $D^{n}_{1,2}$ is the pullback of $\lambda$ along a map $\M_{0,n} \rightarrow \M_g$, and hence this linear series maps $\M_{0,n}$ to the Satake compactification $A_{g}^{*}$ \cite{AGS}.  The image of the morphism associated to the divisor $D^{n}_{1,n/2}$ is the GIT quotient $(\Pro^{1})^{n} \dblq \SL(2)$ with the symmetric linearization \cite{AGS}.  These examples piqued our interest in understanding the geometry of the morphisms associated to the divisors $D^{n}_{1,j}$; these morphisms have been studied further by Giansiracusa (cf.  Remark \ref{Giansiracusa} and \cite{Giansiracusa}).

In a related paper, Alexeev, Gibney, and Swinarski study conformal blocks divisors $D^{2}_{\ell,1}$ for $\sL_2$, level $\ell$, and weights $(1,\ldots,1)$ when $n$ is even \cite{AGS}.  We compare and contrast that paper and this one.

One aspect of the divisors in the families $\{ D^{n}_{1,j} \}$ and $\{ D^{2}_{\ell,1} \}$ that makes them amenable to study is that we have good understanding of the ingredients which are needed to use Fakhruddin's formulas for Chern classes.  Specifically, the \emph{fusion rules}, which give the ranks of the conformal blocks for $n=3$ marked points, are combinatorial and relatively easy if $\mathfrak{g}=\sL_2$ or if $\mathfrak{g}=\sL_k$ and $\ell=1$.  For general simple Lie algebras and levels, the fusion rules can be much more complicated, making it correspondingly more difficult to obtain general formulas.


In order to study the $\sL_2$ divisors in \cite{AGS}, we needed to find a useful formula for the intersection of the divisors with $\F$-curves.  This required application of conformal blocks tools such as factorization.  On the other hand, for the $\sL_n$ level $1$ divisors, Fakhruddin gives a combinatorial formula for the intersection of the divisors with $\F$-curves \cite{Fakh}*{5.2}.  

In \cite{AGS} it is shown that the $\sL_2$ divisors $D^{2}_{\ell,1}$ lie on the boundary of the symmetric nef cone, and that for all $n \geq 10$, four of these divisors span extremal rays of the symmetric nef cone.  It was relatively easy to find sets of curves which  intersect these divisors in degree zero and to prove independence of these families.  In contrast, all of the $\sL_n$ divisors $D^{n}_{1,j}$, $j=2,\ldots,\lfloor n/2 \rfloor$ span extremal rays of the symmetric nef cone.  Finding families of curves which these divisors intersect in degree zero and proving the independence of these families of curves was considerably more difficult in the $\sL_n$ case.  So, although the divisors considered in the present paper and in \cite{AGS} are closely related, the types of results we needed to prove were different for the two families.

\textit{Outline of the paper:}  In Section \ref{background} we establish the notation and conventions we will follow, giving basic definitions and references about conformal blocks bundles and $\M_{0,n}$.  In Section \ref{slnmorphisms} we prove Theorem \ref{thmA} on factoring morphisms through Hassett spaces; this section may be may be read independently of the rest of the paper.  In Section \ref{newtools} we introduce some new tools for working with curves and divisors on $\M_{0,n}$; the main result here is that we invert the matrix of intersection numbers between natural bases of symmetric divisors and symmetric curves on $\M_{0,n}$.  This allows us to write divisors in the natural basis given their intersection numbers with $\F$-curves.  Likewise, we may write an arbitrary symmetric $\F$-curve class in the natural basis of symmetric curve classes.  In Section \ref{divisor classes} we give explicit formulas for the intersection numbers and divisor classes of $D^{n}_{1,j}$.   In Section \ref{special cases section} we show that $D^{n}_{1,j}$ is extremal in the symmetric nef cone if $j=1, 2,3,4$ or if $j | n$.  In Section \ref{k geq 3 structure} we define and study the family $\mathcal{G}$ which is used to show the extremality of $D^{n}_{1,j}$ when $\lfloor n/j \rfloor \geq 3$, and in Section \ref{k geq 3 independence} we prove the extremality of $D^n_{1,j}$ in this case.  We outline a parallel analysis for the case $\lfloor n/j \rfloor =2$ in Sections \ref{k=2 structure} and \ref{k=2 independence}.  Examples and the \texttt{Macaulay 2} code used to compute them are provided in each case; see  \neturltilde{http://www.math.uga.edu/~davids/agss/index.html}{http://www.math.uga.edu/$\sim$davids/agss/index.html} or follow the hyperlinks in this .pdf file.

\subsection*{Acknowledgements}
This project was initiated during Spring Semester 2009 under the Vertical Integration
of Research and Education (VIGRE) Program sponsored by the National
Science Foundation (NSF) at the Department of Mathematics at the University
of Georgia (UGA). We would like to acknowledge the NSF VIGRE grant DMS-03040000
 for partial financial support of this project.  In addition to the authors, the Spring 2009 UGA VIGRE Algebraic Geometry group included Tyler Kelly, David Krumm, Lev Konstantinovsky and Brandon Samples.  We would also like to acknowledge  the participants of the UGA Conformal Blocks Seminar, which in addition to the authors included Valery Alexeev, Brian Boe, Bill Graham, Elham Izadi, and Robert Varley.  Finally we would like to thank Noah Giansiracusa, Daniel Krashen, Chris Manon, and Michael Thaddeus for helpful conversations.

\section{Background: Conformal blocks, \texorpdfstring{$\M_{0,n}$}{M\_0,n}, cones of divisors and curves}\label{background}
In this section we give the basic notation and definitions we use and give references on conformal blocks, the moduli space of curves and basic notions in birational geometry.  We claim no originality for any of the results discussed below; we hope that this information will be useful to algebraic geometers and conformal blocks readers.

\subsection{Conformal blocks bundles} \label{introsection}
 We refer the reader to  the papers \cite{Beauville}, \cite{Faltings}, \cite{Fakh}, \cite{Looijenga}, \cite{Ueno}, for details and proofs of the statements and facts about conformal blocks which we use in the sequel.

We begin with an informal description of the conformal blocks vector bundles.  Let $\mathfrak{g}$ be a Lie algebra.  Let $\ell \in \N$ be a positive integer, called the \textit{level}.  
Let $\w = (\lambda_{1}, \ldots, \lambda_{n})$ be an $n$-tuple of dominant integral weights in the Weyl alcove of level $\ell$ in the root system for $\mathfrak{g}$.

Let $\widehat{\mathfrak{g}}$ be the affine Lie algebra associated to $\mathfrak{g}$:   
\[  \widehat{\mathfrak{g}} = (\ \mathfrak{g} \otimes \C((z_{i})) )\oplus \C c.
\]

Just as in the case of finite-dimensional Lie algebras, to each weight $\lambda_{i}$ we may associate an irreducible  $\widehat{\mathfrak{g}}$-module $\mathcal{H}_{\lambda_{i}}$.  Let $\mathcal{H}_{\w}$ be the tensor product of these modules for the weights in $\w$.

We can describe  the fiber of the conformal blocks bundle $\VV{\mathfrak{g}}{\ell}{\w}$ over a point $(C, p_{1},\ldots, p_n) \in M_{g,n}$ as follows:  Let $U = C-\{p_{1}, \ldots, p_{n} \}$.  Choose a local coordinate $z_{i}$ near each $p_{i}$; this gives ring homomorphisms $\mathcal{O}(U) \rightarrow \C((z_{i}))$.

Define a $\mathfrak{g} \otimes \mathcal{O}(U)$ action on $\mathcal{H}_{\w}$ as follows:
\[   (X \otimes f) \cdot (v_{1} \otimes \cdots \otimes v_{n}) = \sum_{i=1}^{n}  v_{1} \otimes \cdots \otimes (X \otimes f_{p_{i}}) v_{i} \otimes \cdots \otimes v_{n}
\]
Then the fiber of $\VV{\mathfrak{g}}{\ell}{\w}$ over the point $(C, p_{1},\ldots, p_n)$ is the vector space of coinvariants of this $\mathfrak{g} \otimes \mathcal{O}(U)$ action on $\mathcal{H}_{\w}$.

While $\mathfrak{g} \otimes \mathcal{O}(U)$ and $\mathcal{H}_{\w}$ are infinite-dimensional, the vector space of coinvariants is finite-dimensional.  These fibers form an algebraic vector bundle on $M_{g,n}$.  (In particular, the dimension of the vector space of coinvariants does not depend on the curve $C$ or the points $\{ p_{i} \}$, but only on the genus $g$ and the number of marked points $n$.)

The construction can be extended to nodal curves, yielding an algebraic vector bundle $\VV{\mathfrak{g}}{\ell}{\w}$ on $\M_{g,n}$.  Here we record a few additional facts about these bundles which are most relevant to our paper. 

\begin{enumerate}
\item The fibers over nodal curves admit very specific direct sum decompositions known as the \emph{factorization rules}.
\item The vector bundle $\VV{\mathfrak{g}}{\ell}{\w}$ on $\Mgn$ admits a projectively flat connection.  
\item When $g=0$, the connection satisfies the KZ equations.  
\item When $g=0$, the vector bundle $\VV{\mathfrak{g}}{\ell}{\w}$ is globally generated, and hence its determinant line bundle is nef.
\end{enumerate}

\subsection{Symmetric divisors and curves on \texorpdfstring{$\M_{0,n}$}{M\_0,n}}

\subsubsection{Divisors on \texorpdfstring{$\M_{0,n}$}{M\_0,n}}
The symmetric group $S_{n}$ acts on $\M_{0,n}$ by permuting the order of the marked points.  In this paper we shall almost exclusively work with $S_n$-symmetric divisors, and so we take a moment to recall the basic tools for these.  A standard reference is \cite{KeelMcKernanContractible}.

The boundary of $\M_{0,n} $ is the locus  parametrizing nodal curves.  The irreducible components of the boundary are denoted $\Delta_{I,I^{c}}$ and are indexed by partitions of the set $\{1,\ldots,n\}$ into two subsets $I,I^{c}$ each of size at least two.  Then $\Delta_{I,I^{c}}$ is the closure of the locus in the moduli space parametrizing curves with two irreducible components meeting at a single node, having  points with labels in $I$ on one component, and points with labels in $I^c$ on the other component.  We write $\delta_{I}$ for the class of the divisor $\Delta_{I,I^{c}}$.  The boundary classes $\{ \delta_{I} \}$ span $\Pic(\M_{0,n})$; the relations between them can be found in \cite{KeelIntersection}.

Let $\Pic(\M_{0,n})_{\Q} :=  \Pic(\M_{0,n}) \otimes \Q$, and let $\Pic(\M_{0,n})_{\Q}^{S_{n}}$ denote the vector subspace of $S_{n}$-symmetric divisor classes on $\M_{0,n}$.    Let $2 \leq k \leq \lfloor \frac{n}{2} \rfloor$.  The divisors $B_{k} = \sum_{|I|=k} \delta_{I}$ are $S_{n}$-symmetric and form a basis of  $\Pic(\M_{0,n})_{\Q}^{S_{n}}$.  They span $\Pic(\M_{0,n})_{\Q}^{S_{n}}$ because they are the images of the $\Delta_I$ under symmetrization.  One easy way to see that they are independent is to compute their intersection numbers with curves in the moduli space; see for instance \ref{inverseintersectionmatrixlemma}.  

We define $g$ in the following way: If $n$ is even, then $g = n/2 -1$, so that $n=2g+2$.  If $n$ is odd, then $g=\lfloor n/2 \rfloor-1$, so that $n=2g+3$.    These choices for $g$ are convenient because we have $\dim \Pic(\M_{0,n})_{\Q}^{S_{n}} = g$ in both cases.   

\subsubsection{Curves on $\M_{0,n}$}\label{curvessection}  There is a set of rational curves on $\M_{0,n}$ called \textbf{$\F$-curves} which are especially useful and important.  Collectively, $\F$-curves sweep out the locus in $\M_{0,n}$ corresponding to curves having at least $n-4$ nodes.  They can be defined combinatorially in terms of their dual graphs.  Let $G$ be a tree with $n$ labelled leaves that is trivalent at every vertex except one, where it is 4-valent.  Then the locus  in $\M_{0,n}$ parametrizing curves with dual graph $G$ is an open rational curve, given by the cross ratio of the four special points on the 4-valent vertex.  The closure of this locus is called an $\F$-curve, in honor of both Faber and Fulton.  Deleting the 4-valent vertex of the graph $G$ partitions the leaves $\{1,\ldots,n\}$ into four nonempty subsets $I_{1}$, $I_{2}$, $I_{3}$, $I_{4}$, and the homology class of the $\F$-curve depends only on this partition.  We denote this $F_{I_1,I_2,I_3,I_4}$.  The classes of $\F$-curves span $H_{2}(\M_{0,n},\Q)$.

Our primary use of $\F$-curves will be in intersecting them with symmetric divisor classes on $\M_{0,n}$.    In particular,  we do not need to know the partition $\{1,\ldots,n\} = I_{1} \amalg I_{2} \amalg I_{3} \amalg I_{4}$; we only need to know the cardinalities $\#I_{1}$, $\#I_{2}$, $\#I_{3}$, $\#I_{4}$.  We call these four numbers the \emph{shape} of the partition. (Given two $\F$-curve classes whose partitions have the same shape, the intersection with the class of $D$ will be the same, since $D$ is symmetric.)    Thus, a partition $a+b+c+d=n$ of the integer $n$ into four positive integers determines an $\F$-curve class, up to $S_{n}$ symmetry.  We interpret the symbol $F_{a,b,c,d}$ as a weighted average of $\F$-curve classes over all partitions of shape $a,b,c,d$.  In the sequel we often omit the fourth index $d=n-a-b-c$ and write $F_{a,b,c} = F_{a,b,c,d}$.  

A key intersection formula we need is due to Keel and McKernan:
\begin{lemma}[\cite{KeelMcKernanContractible}*{Cor. 4.4}] \label{Keel McKernan formula} Let $D = \sum_{k=2}^{\lfloor \frac{n}{2} \rfloor} \alpha_{k} B_{k}$ be a symmetric divisor class on $\M_{0,n}$.  Let $a+b+c+d=n$ be a partition of $n$ into 4 positive parts.  Then 
\[   D \cdot F_{a,b,c,d} = -\alpha_{a}-\alpha_{b}-\alpha_{c}-\alpha_{d}+\alpha_{a+b}+\alpha_{a+c}+\alpha_{a+d}
\]
where if $t>\lfloor \frac{n}{2} \rfloor$, we define $\alpha_{t} = \alpha_{n-t}$.
\end{lemma}

\subsubsection{Cones of curves and divisors}\label{symmetricFcone}
The cohomology ring of $\M_{0,n}$ is well-understood; it was first described by  Keel \cite{KeelIntersection}.  However the birational geometry of $\M_{0,n}$ is not completely understood and is known to be quite complicated.  An important open question in algebraic geometry is to determine the nef cone of $\M_{0,n}$.  

Recall that a divisor $D$ on a space $X$ is \emph{nef} if $D \cdot C \geq 0$ for every effective curve $C\subset X$. The nef cone is the closure of the ample cone --- those divisors that (after taking suitably large multiples) yield embeddings of $X$ in projective space.  

We define a cone $F_{0,n} \subset \Pic(\M_{0,n})_{\R}$ called the $\F$-cone as the cone of divisors which nonnegatively intersect the $\F$-curves:
\begin{equation}  F_{0,n} := \{ D \in \Pic(\M_{0,n})_{\R} \mid D \cdot F_{I_1,I_2,I_3,I_4} \geq 0 \mbox{ for all partitions $\{1,\ldots,n\} = I_{1} \amalg I_{2} \amalg I_{3} \amalg I_{4}$ } \}.
\end{equation}
In particular, every nef divisor must intersect the $\F$-curves nonnegatively, so we have an inclusion $\Nef(\M_{0,n}) \subset F_{0,n}$.  For $n\leq 7$, this inclusion is an equality, but at the time of this writing, it is not known whether this inclusion is an equality for $n >7$ \cite{Larsen}.

The cone $F_{0,n}$ is combinatorially defined but intractable for large values of $n$; the Picard number of $\M_{0,n}$ is $2^{n-1} - \binom{n}{2}-1$ \cite{KeelIntersection}*{p. 550}.  Thus for $n \geq 7$ it is not practical to experiment with the full cone $F_{0,n}$.  This is one reason we are led to consider a smaller cone: Let 
\begin{equation}
\widetilde{F}_{0,n} := F_{0,n} \cap \Pic(\M_{0,n})_{\R}^{S_n}.
\end{equation}
Then $\widetilde{F}_{0,n}$ contains the cone $\SymNef(\M_{0,n})$ of $S_n$-symmetric nef divisors, and is much simpler since its ambient space is only $g$-dimensional, and there are fewer partitions of integers $a+b+c+d=n$ than partitions of the set $\{1,\ldots,n\}$.

Studying $\widetilde{F}_{0,n}$ sheds light not only on $\M_{0,n}$, but also on $\M_{g}$: Gibney, Keel, and Morrison showed that the birational geometry of $\M_g$ is determined by the cone $\widetilde{F}_{0,g}$ associated to $\M_{0,g}$ \cite{GibneyKeelMorrison}.

In this paper we will study nef divisors that span extremal rays of the cone $\widetilde{F}_{0,n}$.  Recall that the ambient vector space $\Pic(\M_{0,n})_{\R}^{S_n}$ of $S_{n}$-symmetric divisors is $g$-dimensional.  Let $\mathcal{F}$ be a family of $g-1$ independent $\F$-curves $\mathcal{F}(i)$.  Then $\mathcal{F}$  determines a ray in $\Pic(\M_{0,n})_{\R}^{S_n}$ given by the intersection of the hyperplanes $\mathcal{F}(i) \cdot D = 0$.    If this ray is spanned by a divisor which intersects all the $\F$-curves nonnegatively, then this ray must be extremal in the $\F$-cone $\widetilde{F}_{0,n}$.  We also know that the divisors $D^{n}_{1,j}$, being conformal blocks divisors, are nef.  Thus, since they are nef and span extremal extremal rays of $\widetilde{F}_{0,n}$, they span extremal rays of $\SymNef(\M_{0,n})$.

Here is an example:

\begin{example}
Let $\mathcal{F} = \{ F_{1,1,j} \mid 1 \le j \le g-1 \}$. The independence of all the curves in the family $\mathcal{F}$ will be shown in Section \ref{newtools}.  

We will see in Section \ref{special cases section} that the divisor $D^n_{1,2}$ intersects every curve in $\mathcal{F}$ in degree zero.  Therefore, this divisor spans an extremal ray in the symmetric nef cone.
\end{example}

\section{\texorpdfstring{$\sL_n$ level $1$}{sl\_n level 1} conformal blocks and Hassett's spaces} \label{slnmorphisms}

 In this section we show that to any conformal blocks divisor $D^{n}_{1,\w}$, one may associate a  minimal Hassett space $\M_{0,\mathcal{A}}$ such that
the morphism defined by the linear series given by $D^{n}_{1,\w}$ factors through the birational contraction $\rho_{\mathcal{A},N}: \M_{0,n} \rightarrow \M_{0,\mathcal{A}}$.   When the weights are symmetric, we also show that the associated morphism factors through a particular log canonical model of $\M_{0,n}$.  Fakhruddin has proved an analogous statement for the conformal blocks divisors given by $\sL_{2}$ \cite{Fakh}*{Prop. 4.7}.

\textit{Notation.}  We refer to Hassett's paper \cite{HassettWeighted} for the definition of the moduli spaces of weighted pointed stable curves $\M_{0,\mathcal{A}}$ with weights $\mathcal{A}$.  We write $N= (1,\ldots,1)$.  In \cite{HassettWeighted}*{Theorem 4.1}, Hassett defines natural birational morphisms  $\rho_{\mathcal{A}, N}: \M_{0,n} \rightarrow \M_{0,\mathcal{A}}$ which may be described as follows:  given a point $[(C,p_1,\ldots,p_n)] \in \M_{0,n}$, its image under $\rho_{\mathcal{A}, N}$ is obtained by successively collapsing components of $C$ along which $K_{C}+\sum_{i=1}^n a_i p_i$ fails to be ample.

\begin{definition}\label{slnHassettSpace}  Let $\w=(\omega_{j_1},\ldots, \omega_{j_n})$ be an $n$-tuple of fundamental dominant weights of $\sL_n$ satisfying  $j_i \in \{1,\ldots,\lfloor \frac{n}{2} \rfloor\}$ for each $i \in \{1,\ldots,n\}$, and $\sum_{i=1}^{n} j_i \geq 2n$.  
Let $D^{n}_{1,\w}$ be the conformal blocks divisor for $\sL_{n}$, level $1$, and weights $\w$.   Let $\mathcal{A}$ be the set of weights $(\frac{j_1}{n},\ldots,\frac{j_n}{n})$.  We say that $\M_{0,\mathcal{A}}$  is the \emph{ minimal Hassett space determined by $D^n_{1,\w}$} and call $\rho_{\mathcal{A},N}: \M_{0,n} \rightarrow \M_{0, \mathcal{A}}$  the \emph{associated maximal contraction}.
\end{definition}

To check that $\mathcal{A}$ determines a Hassett weighted moduli space, one needs to know that $\frac{j_i}{n} \in (0,1]$ for each $i \in \{1,\ldots,n\}$, and that $\sum_{i=1}^{n}\frac{j_i}{n} \ge 2$.   
The hypotheses on $\w$ ensure that these conditions on $\mathcal{A}$ are satisfied. 

 \begin{theorem}\label{factor}   Let $\w=(\omega_{j_1},\ldots, \omega_{j_n})$ be an $n$-tuple of fundamental dominant weights of $\sL_n$ satisfying  $j_i \in \{1,\ldots,\lfloor \frac{n}{2} \rfloor\}$ for each $i \in \{1,\ldots,n\}$, and $\sum_{i=1}^{n} j_i \geq 2n$.   Let $\M_{0,\mathcal{A}}$ be the minimal Hassett space determined by $D^n_{1,\w}$.  Then the morphism $\M_{0,n} \rightarrow X$ associated to the divisor $D^n_{1,\w}$ factors through the natural birational contraction $\rho_{\mathcal{A},N}:\M_{0,n} \rightarrow \M_{0,\mathcal{A}}$ defined by Hassett.     
\end{theorem}


 In the proof of Theorem \ref{factor}, we shall see that if $F=F_{N_1,N_2,N_3,N_4}$ is any $\F$-curve that is contracted by $\rho_{\mathcal{A},N}$, then $D^n_{1,\omega} \cdot F =0$.
For the reader's convenience, we state just what we need of Fakhruddin's Proposition $5.2$ in order to show that the divisors $D^n_{1,\w}$ contract the same $\F$-curves as the Hassett contractions $\rho_{\mathcal{A},N}$.

\begin{proposition}\label{Fakh5.2} \cite{Fakh}*{Prop. 5.2} 
Let $\w=(\omega_{i_1}, \omega_{i_2}, \ldots, \omega_{i_n})$ with $0 \le i_j < n$ for $j=1,2,\ldots, n$, where $\omega_{0}=0$.  Let $\F$
be an $\F$-curve on $\M_{0,n}$ corresponding to a partition $\{1,2,\ldots, n\}=\cup_{k=1}^{4}N_k$.  Let $\nu_k$ be the representative in $\{0,1,\ldots, n-1\}$ of $\sum_{j \in N_k}i_j$ modulo $n$.
If $D^{n}_{1,\w} \cdot F \ne 0$, then $\sum_{k} \nu_k=2n$.
\end{proposition}

 \begin{proof}[Proof of Theorem \ref{factor}]
Let $f^n_{1,\w}: \M_{0,n} \rightarrow Y$ be the morphism given by $D^n_{1,\w}$ and $\rho_{\mathcal{A},N}$ be the contraction of $\M_{0,n}$ onto the minimal Hassett moduli space  $\M_{0,\mathcal{A}}$.  To show that $f^n_{1,\w}$ factors through $\rho_{\mathcal{A},N}$ we would like to define a morphism $f:\M_{0,\mathcal{A}} \rightarrow Y$, such that $f^n_{1,\w}=f \circ \rho_{\mathcal{A},N}$.  Set theoretically, this will be the map
$f(y)=f^n_{1,\w}(\rho_{\mathcal{A},N}^{-1}(y))$.  This is well defined as long as whenever $X$ is a curve on $\M_{0,n}$ that is contracted by $\rho_{\mathcal{A},N}$, then $X$ is also contracted by $f^n_{1,\w}$.  We show that this
is true for $\F$-curves.  Since every effective curve that is contracted by $\rho_{\mathcal{A},N}$ is a linear combination of the $\F$-curves that get contracted (\cite{Fakh}, Lemma $4.6$), the map is well defined.  It remains to check that the intersection $D^n_{1,\w} \cdot F_{N_1,N_2,N_3,N_4} = 0$ for all $\F$-curves $F_{N_1,N_2,N_3,N_4}$ that are contracted by  $\rho_{\mathcal{A},N}$.  

An $\F$-curve $F=F_{N_1,N_2,N_3,N_4}$ is contracted by $\rho_{\mathcal{A},N}$ precisely when $\sum_{i \in N_1 \cup N_2 \cup N_3} \frac{j_i}{n} \le 1$.  (Here we order the indices so that $|N_1| \le |N_2| \le |N_3| \le |N_4|$.)    Using Fakhruddin's notation from Proposition \ref{Fakh5.2}, this gives $\sum_{i =1}^3 \nu_i  \le n$ and by definition, $\nu_4 <n$.  In order for the intersection $D^n_{1, \w} \cdot F$ to be nonzero, the sum $\sum_{i=1}^4 \nu_i$ must be $2n$, which is impossible.

\end{proof}

\begin{remark}\label{Giansiracusa}  Typically, the divisors $D^n_{1,j}$ contract more curves than the contraction maps $\rho_{\mathcal{A},N}$ from $\M_{0,n}$ to the minimal Hassett space $\M_{0,\mathcal{A}}$ determined by $D^n_{1,j}$.  
For instance, the morphism of the divisor $D^n_{1,2}$ has image equal to $(\mathbb{P}^1)^n\dblq SL_2$ (cf. \cite{AGS}).     
In \cite{GiansiracusaSimpson}, Giansiracusa and Simpson study moduli spaces of pointed conics $Con(n) \dblq \SL_3$ and describe birational morphisms  $\overline{\mathcal{M}}_{0,n} \rightarrow \mathtt{Con}(n) \dblq \SL_3$.  While talking with Giansiracusa, we noticed that for certain linearizations,  the curves that get contracted under the morphisms that he and Simpson study are the same curves that get contracted by $D^n_{1,3}$.   Giansiracusa has generalized this observation further, identifying the images of the morphisms given by the $D^n_{1,j}$ with GIT quotients (\cite{Giansiracusa}).
\end{remark}

\section{New tools for working with curves and divisors on \texorpdfstring{$\M_{0,n}$}{M\_0,n}}\label{newtools} 
The main objects of study in this work are certain semiample divisors that come from conformal blocks.  One of the ways we can study these divisors is by understanding their intersections with families of independent $\F$-curves in $H_{2} (\M_{0,n} ,\Q)$.  We give an example below.  

 \begin{proposition}\label{C1} Let $n=2g+2$, or $n=2g+3$.  The set of curves $\mathcal{F}=\{F_{1,1,i}: 1 \le i \le g\}$ is independent.
 \end{proposition}
(We claim no originality for this result.  It appears, for instance, in  \cite{GibneyNumerical} and 
\cite{AlexeevSwinarski}.  We include a proof because it illustrates
our techniques well.  See the related paper \cite{AGS} for a proof by a different method.)

We prove Proposition \ref{C1} by studying matrices of intersection numbers, and answer the following question about divisors at the same time:  Suppose we have a symmetric divisor $D$ on $\M_{0,n}$.  Given the intersection numbers of $D$ with the $\F$-curves $\{ F_{i,1,1,n-i-2} \}$, how can we write $D$ in the basis of $\Pic(\M_{0,n})_{\Q}^{S_n}$ consisting of the divisors $B_{2}, \ldots, B_{\lfloor n/2 \rfloor}$?

Our goal is to find coefficients $b_j$ such that $D = \sum_{i=2}^{g+1} b_{i} B_{i}$.  (Note the indexing on the right: there are $g$ divisors, starting with $B_{2}$ and ending with $B_{g+1}$.)

Each curve $F_{i,1,1}$ yields an equation:
\[   D \cdot F_{i,1,1} = (b_{2} B_{2} + b_{3} B_{3} + \cdots + b_{g+1} B_{g+1} ) \cdot F_{i,1,1} = b_{2} (B_{2} \cdot F_{i,1,1} )+ b_{3} ( B_{3}\cdot F_{i,1,1} ) + \cdots +  b_{g+1} (B_{g+1}\cdot F_{i,1,1} )
\] 

We can write this as a matrix equation:
\begin{displaymath}
\left( \begin{array}{cccc}
B_{2} \cdot F_{1,1,1}  & B_{3} \cdot F_{1,1,1} & \cdots & B_{g+1} \cdot F_{1,1,1} \\
B_{2} \cdot F_{2,1,1}  & B_{3} \cdot F_{2,1,1} & \cdots & B_{g+1} \cdot F_{2,1,1} \\
\vdots & \vdots & \cdots & \vdots\\
B_{2} \cdot F_{g,1,1}  & B_{3} \cdot F_{g,1,1} & \cdots & B_{g+1} \cdot F_{g,1,1}
\end{array} \right)
\left( \begin{array}{c}
b_{2}\\
b_{3}\\
\vdots\\
b_{g+1}
\end{array} \right)
= 
\left( \begin{array}{c}
D \cdot F_{1,1,1}\\
D \cdot F_{2,1,1}\\
\vdots \\
D \cdot F_{g,1,1}\\
\end{array}
\right).
\end{displaymath}

Let $M$ be the matrix of intersection numbers $(F_{i,1,1} \cdot B_{j})$, and let $N:=M^{-1}$.  Then we have 
\begin{displaymath}
\left( \begin{array}{c}
b_{2}\\
b_{3}\\
\vdots\\
b_{g+1}
\end{array} \right)
= N 
\left( \begin{array}{c}
D \cdot F_{1,1,1}\\
D \cdot F_{2,1,1}\\
\vdots \\
D \cdot F_{g,1,1}\\
\end{array}
\right).
\end{displaymath}

The following lemma, whose proof is an easy exercise, gives a formula for $N$.  

\begin{lemma} \label{inverseintersectionmatrixlemma} If $n$ is odd, then
\begin{displaymath}
N_{rs} = \left\{\begin{array}{l}
\rule{0pt}{18pt} \displaystyle  \frac{r(r+1)}{2(g+1)} - r+s \qquad \mbox{if $s < r$}\\
\rule{0pt}{18pt} \displaystyle  \frac{r(r+1)}{2(g+1)} \qquad \mbox{if $s\geq r$}
\end{array} \right.
\end{displaymath}
If $n$ is even, then:
\begin{displaymath}
N_{rs} = \left\{\begin{array}{l}
\rule{0pt}{18pt} \displaystyle\frac{r(r+1)}{(2g+1)} -r+s\qquad \mbox{if $s<g, s < r$}\\
\rule{0pt}{18pt}\displaystyle \frac{r(r+1)}{(2g+1)} \qquad \mbox{if $s<g, s \geq r$}\\
\rule{0pt}{18pt} \displaystyle \frac{r(r+1)}{2(2g+1)} \qquad \mbox{if $s=g$}\\
\end{array} \right.
\end{displaymath}
\end{lemma}

\begin{example}  Let $n=13$. Then 
\begin{displaymath}%
M = \left(
\begin{array}{rrrrr}
3 & -1 & 0 & 0 & 0\\
0 & 2 & -1 & 0 & 0 \\
1 & -1 & 2 & -1 & 0 \\
1 & 0 & -1 & 2 & -1 \\
1 & 0 & 0 & -1 & 1 \\
\end{array} \right), \qquad \qquad N = \frac{1}{6}
\left(
\begin{array}{rrrrr}
1 & 1 & 1 & 1 & 1 \\
-3 & 3 & 3 & 3 & 3 \\
-6 & 0 & 6 & 6 & 6 \\
-8 & -2 & 4 & 10 & 10 \\
-9 & -3 & 3 & 9 & 15
\end{array} \right).
\end{displaymath}
\end{example}

\begin{example} \label{n=12 example} Let $n=12$.  Then 
\begin{displaymath}
M = \left(
\begin{array}{rrrrr}
3 & -1 & 0 & 0 & 0\\
0 & 2 & -1 & 0 & 0 \\
1 & -1 & 2 & -1 & 0 \\
1 & 0 & -1 & 2 & -1 \\
1 & 0 & 0 & -2 & 2
\end{array} \right), \qquad \qquad N = \frac{1}{11}
\left(
\begin{array}{rrrrr}
2 & 2 & 2 & 2 & 1\\
-5 & 6 & 6 & 6 & 3 \\
-10 & 1 & 12 & 12 & 6\\
-13 & -2 & 9 & 20 & 10\\
-14 & -3 & 8 & 19 & 15
\end{array} \right).
\end{displaymath}
\end{example}

In particular, since $M$ is invertible, it has full rank, and this proves Proposition \ref{C1}.

\subsection{Writing $\F$-curves \texorpdfstring{$F_{a,b,c,d}$}{F\_a,b,c,d} in the basis \texorpdfstring{$\{ F_{j,1,1} \}$}{\{F\_j,1,1\}}}

Since the $\F$-curves $\{ F_{j,1,1} \}_{j=1}^{g}$ form a basis for $H_{2}(\M_{0,n},\Q)^{S_n}$, we may write $F_{a,b,c,d}  = \sum_{j=1}^{g} \gamma_j F_{j,1,1}$.  In this subsection, we obtain a general formula for the coefficients $\gamma_j$.  

Let $P := N^{t}$ be the transpose of the matrix $N$ of Lemma \ref{inverseintersectionmatrixlemma}: 

\begin{equation} \label{matrix P}
P_{s,t} = \left\{ \begin{array}{ll}
\rule{0pt}{18pt} \frac{t(t+1)}{n-1} - t + s & \mbox{if $s<g$ and $s<t$} \\
\rule{0pt}{18pt} \frac{t(t+1)}{n-1}  & \mbox{if $s<g$ and $s\geq t$} \\
\rule{0pt}{18pt} \frac{t(t+1)}{n-1}  & \mbox{if $s=g$ and $n$ odd} \\
\rule{0pt}{18pt} \frac{t(t+1)}{2(n-1)}  & \mbox{if $s=g$ and $n$ even} \\
\end{array} \right.
\end{equation}

Suppose $\sum_{j=1}^{g} \gamma_j F_{j,1,1} = F_{a,b,c,d}$.  We can intersect both sides of this equation 
with the divisors $B_{2},\ldots,B_{g+1}$ and solve for
the vector  $\vec{\gamma} $ of numbers $\gamma_{j}$:
\begin{displaymath} \vec{\gamma} = P \left( \begin{array}{c}
B_2 \cdot F_{a,b,c,d}\\
B_3 \cdot F_{a,b,c,d}\\
\vdots \\
B_{g+1} \cdot F_{a,b,c,d}\\
\end{array}
\right).
\end{displaymath}

By combining this formula for $P$ with Keel and McKernan's formula for intersection numbers (see \ref{Keel McKernan formula}) above, we obtain the following formula for $\gamma_{j}$.  Note that the function $A$ below arises from combining all the fractional parts of \ref{matrix P}; the function $B$ arises from the remaining terms in line 1 of \ref{matrix P}.

\begin{proposition} \label{gamma j}  Suppose that $a \leq b \leq c \leq d$.  
Define three piecewise linear functions $A$, $B$, $f$ as follows:
\begin{eqnarray}
A(a,b,c,d) & :=  &\left\{ \begin{array}{ll}
0 & \mbox{if $d > g+1$} \\
2a & \mbox{if $a+d \leq g+1$} \\
n-2d & \mbox{if $d \leq g+1$ and $a+c \leq g+1 < a+d$} \\
2b & \mbox{if $d \leq g+1$ and $a+b \leq g+1 < a+c$} \\
n-2a & \mbox{if $d \leq g+1$ and $g+1 < a+b$} \\
\end{array} \right. \nonumber \\
f(x,n) & := & \left\{ \begin{array}{ll}
x & \mbox{if $x \leq g+1$} \\
n-x & \mbox{if $x > g+1$}
\end{array} \right. \nonumber \\
B(j,a,b,c,d) & := & \sum_{t \in \{a,b,c,d\}} \max\{ (f(t)-1-j), 0 \} -\sum_{ u \in \{b,c,d\}}  \max \{  f(a+u)-1-j, 0 \}
\end{eqnarray}

Then $\gamma_j$, the coefficient on $F_{j,1,1}$ when $F_{a,b,c,d}$ is written in the $\{F_{j,1,1}\}$ basis, is
\begin{equation}
\gamma_j =  \left\{ \begin{array}{ll}
A(a,b,c,d)+B(j,a,b,c,d) & \mbox{if $n$ is odd, or $n$ even, $j \neq g$} \\
\rule{0pt}{18pt} \frac{1}{2}A(a,b,c,d)+B(j,a,b,c,d) & \mbox{if $n$ even, $j =g$} \\
\end{array} \right.
\end{equation}
\end{proposition}

\begin{example}  Let $n= 12$.  Suppose we want to compute $F_{1,2,2}$
  in the basis $\{ F_{j,1,1} \}$.  The matrix $N$ is given in Example
  \ref{n=12 example}.  We compute
\begin{displaymath}
\vec{\gamma} = \frac{1}{11} \left(
\begin{array}{rrrrr}
2 & -5 & -10 & -13 & -14 \\
2 & 6 & 1 & -2 & -3 \\
2 & 6 & 12 & 9 & 8 \\
2 & 6 & 12 & 20 & 19 \\
1 & 3 & 6 & 10 & 15 \\
\end{array} \right)
\left(
\begin{array}{r}
-2 \\
2\\
1\\
-1 \\
0
\end{array} \right)
 = (-1,1,1,0,0)
\end{displaymath}
Thus, 
\[  F_{1,2,2,7} = -F_{1,1,1,9} + F_{2,1,1,8} + F_{3,1,1,7}.
\]
Alternatively, using Proposition \ref{gamma j} above, we find: $A(1,2,2,7)=0$, $B(1,1,2,2,7)=-1$, \\$B(2,1,2,2,7) = B(3,1,2,2,7)=1$, and $B(4,1,2,2,7)=B(5,1,2,2,7)=0$.
\end{example}

\begin{remark}
When $n$ is even, we can show $A(a,b,c,d)$ is even, and it follows that  the numbers $\gamma_j$ are integers.
\end{remark}

\subsection{Some useful formulas}
In this subsection we give more explicit formulas for some curves which appear often in the
sequel.  These cover many, but not all, of the curves we work with; any example not covered by the propositions below can be calculated
from Proposition \ref{gamma j} instead.  The formulas below are easily obtained from Proposition \ref{gamma j}.

\begin{proposition}\label{case 1}
\mbox{} \\
Suppose $2k < g+1-2k$.
Suppose that $i+2k \leq g+1$ and $i \neq g+1-2k$ if $n$ is even.  Then $F_{i,k,k} = \sum_{j=1}^{g} \gamma_{j} F_{j,1,1}$, where the coefficients $\gamma_{j}$ are given below.

If $i<k$:
\begin{equation}
\gamma_j = \left\{ \begin{array}{ll}
-j-1 & \mbox{ if $j<i$}                              \\
-i & \mbox{ if $ i \leq j < k$}                     \\
2j+2-i-2k &\mbox{  if $k \leq j < i+k$}   \\
i & \mbox{ if $i+k \leq j < 2k$ }                 \\
i+2k-j-1& \mbox{ if $2k \leq j < i+2k$ }\\
0 & \mbox{ if $i+2k \leq j$ }                   
\end{array} \right.
\end{equation}

If $k \leq i<2k$:
\begin{equation}
\gamma_j = \left\{ \begin{array}{ll}
 -j-1 & \mbox{ if $j<k$  }                              \\
 j+1-2k & \mbox{ if $k \leq j < i$ }                \\
 2j+2-i-2k& \mbox{ if $i \leq j < 2k$ }        \\
 j+1-i & \mbox{ if $2k \leq j < i+k$  }           \\
 i+2k-j-1 & \mbox{ if $i+k \leq j < i+2k$ }  \\
 0 & \mbox{ if $i+2k \leq j$         }                 
\end{array} \right.
\end{equation}

If $2k \leq i \leq g+1-2k$:
\begin{equation}
\gamma_j = \left\{ \begin{array}{ll}
 -j-1 &\mbox{ if $j<k$} \\
 j+1-2k &\mbox{ if $k \leq j < 2k$} \\
 0&\mbox{ if $2k \leq j < i$}\\
 j+1-i & \mbox{ if $i \leq j < i+k$}\\
i+2k-j-1&\mbox{ if $i+k \leq j < i+2k$}\\
0 &\mbox{ if $i+2k \leq j$ }
\end{array} \right.
\end{equation}

\end{proposition}

\begin{example}
Let $n=62$, $k=7$, $i=16$.  Then the vector $\vec{\gamma}$ of coefficients $\gamma_{j}$ is \\
\small$(-2, -3, -4, -5, -6, -7, -6, -5, -4, -3, -2, -1, 0, 0, 0, 1, 2, 3, 4, 5, 6, 7, 6, 5, 4, 3, 2, 1, 0, 0).$
\normalsize \end{example}

\begin{proposition} \label{case 2}
Suppose that $i+k \leq g+1 < i+2k$, and that $k \leq 2k \leq i \leq n-i-2k \leq i+k$, and $i \neq g+1-k$ if $n$ is even.  

Then $F_{i,k,k} = \sum_{j=1}^{g} \gamma_{j} F_{j,1,1}$, where the coefficients are given as follows.
\begin{equation}
\gamma_j = \left\{ \begin{array}{ll}
-j-1 & \mbox{if $j<k$} \\
-2k+j+1 & \mbox{if $k \leq j < 2k$} \\
0 & \mbox{if $2k \leq j < i$} \\
-i+j+1 & \mbox{if $i \leq j < n-i-2k$} \\
2k+2j-n+2 & \mbox{if $n-i-2k\leq j < i+k$} \\
2i+4k-n & \mbox{if $i+k \leq j \leq g-1$, or if $n$ odd, $j=g$} \\
i+2k-g-1 & \mbox{ if $n$ even, $j=g$}
\end{array} \right.
\end{equation}
\end{proposition}
\begin{example}
Let $n=62$, $k=9$, $i=19$.  Then the vector $\vec{\gamma}$ of coefficients $\gamma_{j}$ is \\ \small$(-2, -3, -4, -5, -6, -7, -8, -9, -8, -7, -6, -5, -4, -3, -2, -1, 0, 0, 1, 2, 3, 4, 5, 6, 8, 10, 12, 12, 12, 6).$
\normalsize \end{example}

We summarize some of the properties of the numbers $\gamma_{j}$ which are used in the sequel.

\begin{proposition} \label{cij summary}  Suppose $j \geq 6$. Given $i,k$, let $\gamma_{j}$ be defined by $F_{i,k,k} = \sum_{j=1}^{g} \gamma_{j} F_{j,1,1}$.  
\begin{enumerate}
\item For $2k-1 \leq i \leq g-k+1$, and for $1 \leq j \leq 2k-2$, the coefficients $\gamma_{j}$ are independent of $i$, and are given by the formulas $\gamma_{j} = -j-1$ if $j<k$, and $j+1-2k$ if $k \leq j \leq 2k-2$.
\item For $2k \leq i \leq g-k+1$, and for $2k-1 \leq j \leq i-1$, the coefficients $\gamma_{j}=0$.
\item For $2k \leq i \leq g-k+1$, the coefficient $\gamma_{i}=1$.
\end{enumerate}
\end{proposition}
\begin{proof}  The formulas from Proposition \ref{case 1} establish these claims when $i \leq g-2k$.  However, Proposition \ref{case 2} does not establish these claims in all cases  when $g-2k+1 \leq i \leq g-k+1$, so we sketch a proof.

We use Proposition \ref{gamma j} to compute $\gamma_j$ in these cases.  When $g-2k+1 \leq i \leq g-k+1$, we have $a=b=k, c=i, d=n-i-2k$, and $a+d \leq g+1$.  Thus $A(k,k,i,n-i-2k) = 2k$.  If $1 \leq j \leq 2k-2$, we have $B(j,k,k,i,n-i-2k) = -j-1-2k$.  If $k-1 \leq j \leq 2k-2$ we have $B(j,k,k,i,n-i-2k)= j+1-4k$.  If $2k-1 \leq j \leq i-2$, we have $B(j,k,k,i,n-i-2k) = 2k$.  If $i-1 \leq j \leq i$, we have $B(j,k,k,i,n-i-2k) = j+1-i$.  The result follows.
\end{proof}

\section{\texorpdfstring{The divisors $D^{n}_{1,j}$: conformal blocks
    for $\sL_{n}$, level $1$, and symmetric weights}{The divisors
    Dn1j: conformal blocks for sl\_n, level 1, and symmetric weights}} \label{divisor classes}
 Fakhruddin has given convenient formulas for computing the intersection of the $D^{n}_{1,j}$ with an arbitrary $\F$-curve.   In Proposition \ref{symm}  we recall  his formula for  intersecting with symmetric $\F$-curves, and in Proposition \ref{intersectionform} we give an explicit expression for the intersection numbers  $a_{\ell j}= D^n_{1,j} \cdot F_{1,1,\ell}$.   The main result that we prove in this section is the following theorem, which gives the classes of the divisors $D^n_{1,j}$.

\begin{theorem}\label{nclass}  Fix $n=2g+2$ or $n=2g+3$ and $j \in \{2,\ldots,g+1=\lfloor \frac{n}{2} \rfloor\}$.  
Write $a_{\ell j} = D^{n}_{1,j} \cdot F_{1,1,\ell, n-\ell-2}$.  Then $D^n_{1,j} \equiv \sum_{r=1}^g b_{rj} B_{r+1}$, where when $n$ is odd, $$b_{rj} = \sum_{\ell=1}^{r-1}\Bigg( \frac{r(r+1)}{2(g+1)}-(r-\ell)\Bigg)a_{\ell j} + \frac{r(r+1)}{2(g+1)} \sum_{\ell =r}^g a_{\ell j}, $$
and when $n$ is even, 
$$b_{rj} = \sum_{\ell =1}^{r-1}\Bigg( \frac{r(r+1)}{2g+1}-(r-\ell)\Bigg)a_{\ell j} + \frac{r(r+1)}{2g+1} \sum_{\ell =r}^{g-1} a_{ \ell j} + \frac{r(r+1)}{2g+1}a_{gj}.$$ 
\end{theorem}
\begin{proof}
The formulas for the numbers $b_{rj}$  follow from Lemma \ref{inverseintersectionmatrixlemma}.  The numbers $a_{ij}$ are given by Proposition \ref{slk l1 numbers} below.    
\end{proof}

\subsection{Fakhruddin's intersection formulas} \label{intersect}
In this subsection we give explicit formulas for the intersection of the $D^n_{1,j}$ with the symmetric $\F$-curves that arise in the expression of the class of $D^n_{1,j}$.

Fakhruddin gives explicit formulas for intersecting $D^n_{1,j}$  with any $\F$-curve on $\M_{0,n}$.    For convenience, we present his intersection formulas for the intersection of $D^n_{1,j}$ with the symmetric $\F$-curves.

\begin{proposition}(Fakhruddin 5.2) \label{symm} \label{fakh 5.2} Let $F_{n_1,n_2,n_3,n_4}$ be the symmetric $\F$-curve defined by the partition $n=\sum_{i=1}^4n_i$ of $n$ into four positive integers.  Let $\nu_i$ be the representative in $\{0,\ldots,n-1\}$ of $jn_i$ modulo $n$.  Put $\nu_{\max}=\max\{\nu_1,\ldots, \nu_4\}$ and $\nu_{\min}=\min\{\nu_1,\ldots,\nu_4\}$.  Then
$$D^n_{1,j} \cdot F_{n_1,n_2,n_3,n_4} = 
\left\{ 
\begin{matrix} 
\nu_{\min}  &  \mbox{if } \ \   \sum \nu_k=2n, \ \ \ \nu_{\max}+\nu_{\min} \le n; \\
n-\nu_{\max}  &  \mbox{if } \ \  \sum \nu_k=2n, \ \ \ \nu_{\max}+\nu_{\min} \ge n; \\
0 & \mbox{otherwise}.
\end{matrix}
\right.$$
\end{proposition}

\textbf{Notation.}  We write $x \, \% \, y$ for the remainder when $x$ is divided by $y$.  Note that this is the unique integer between $0$ and $y-1$ representing the congruence class of $x\bmod y$.

\begin{proposition} \label{when int numbers are zero}  Write $n=2g+2$ or $n=2g+3$.  Let $2 \leq j \leq g+1$.  
\begin{enumerate}
\item If $\sum \nu_k \neq 2n$, then $D^n_{1,j} \cdot F_{n_1,n_2,n_3,n_4} =0$.
\item If $\nu_{i}=0$ for any $i$, then $D^n_{1,j} \cdot F_{n_1,n_2,n_3,n_4} =0$.
\item if $D^n_{1,j} \cdot F_{n_1,n_2,n_3,n_4} =0$, then either $\sum
  \nu_k \neq 2n$, or $\nu_{i}=0$, or both.
\item $D^n_{1,j} \cdot F_{n-i-2,i,1,1} \neq 0 \Longleftrightarrow \exists p \in \Z$ such that $\frac{ij}{n} < p  < \frac{(i+2)j}{n}$.
\end{enumerate}
\end{proposition}
\begin{proof}
The first statement is part of Proposition \ref{fakh 5.2}.  For the
second, note that if  $\sum \nu_k \neq 2n$, then by the first
statement we have $D^n_{1,j} \cdot F_{n_1,n_2,n_3,n_4} =0$, and if
$\sum \nu_k = 2n$, we have $\nu_{\min} = 0$ and $\nu_{\max}<n$, so we
are in the first case of the formula, and $D^n_{1,j} \cdot
F_{n_1,n_2,n_3,n_4} =\nu_{\min} = 0$.   

For the third statement, note that since $\nu_{\max} < n$, there is no
way to get $D^n_{1,j} \cdot F_{n_1,n_2,n_3,n_4} =0$ out of the second
case of the formula.  Hence, if $D^n_{1,j} \cdot
F_{n_1,n_2,n_3,n_4} =\nu_{\min} = 0$, we must be in the first case or
the third case.

The fourth statement requires a little more work.  To compute $\nu_1 =
(n-i-2)j \%n$, we seek $\alpha$ such that $0 \leq \alpha n-ij-2j < n$.
Then $\frac{(i+2)j}{n} \leq \alpha < \frac{(i+2)j}{n}+1$, so $\alpha =
\lceil \frac{(i+2)j}{n} \rceil$.  Observe that $\frac{(i+2)j}{n} \in
\Z \Leftrightarrow \nu_1=0$.

To compute $\nu_2$, we seek $\beta$ such that $0 \leq ij-\beta n <n$.
Then $\frac{ij}{n} \geq \beta > \frac{ij}{n}-1$, so $\beta = \lfloor
\frac{ij}{n} \rfloor$.  Observe that $\frac{ij}{n} \in \Z
\Leftrightarrow \nu_2 =0$.

We have $\nu_3 = \nu_4 = j \neq 0$ and 
\begin{displaymath}
\sum_{i=1}^{4} \nu_i = \left( \left\lceil \frac{(i+2)j}{n} \right\rceil - \left\lfloor
\frac{ij}{n} \right\rfloor \right)n.
\end{displaymath}

Note that $1 \leq \left\lceil \frac{(i+2)j}{n} \right\rceil - \left\lfloor
\frac{ij}{n} \right\rfloor \leq 2$.

By the third statement of the proposition, to have $D^n_{1,j} \cdot F_{n_1,n_2,n_3,n_4}
=0$, we must have either $\nu_1=0$ or $\nu_2=0$ or $\sum \nu_i \neq 2n$.
To have $\lceil \frac{(i+2)j}{n} \rceil - \lfloor
\frac{ij}{n} \rfloor =2$, there must exist an integer $p$ such
that $ \frac{ij}{n} < p < \frac{(i+2)j}{n} $.  Under these circumstances we have $\nu_1 \neq 0$ and $\nu_2 \neq 0$, and  the fourth statement follows.

\end{proof}

The particular intersection numbers $D^n_{1,j} \cdot F_{1,1,i}$ arise in the class of $D^n_{1,j}$.  To make them even more explicit, we define 
\begin{displaymath}   \kappa(n,j,i) := n - (ij \, \% \, n)
\end{displaymath}

\begin{proposition} \label{slk l1 numbers} \label{intersectionform}
For $2 \le j \le \lfloor \frac{n}{2} \rfloor$, 
\begin{displaymath}
D^{n}_{1,j} \cdot F_{1,1,i,n-i-2} =  \left \{ \begin{array}{cl} 
\kappa(n,j,i)         &    \lfloor\frac{ij}{n} \rfloor + \lfloor\frac{(n-i-2)j}{n}\rfloor = j-2, \mbox{and }1 \le \kappa(n,j,i) \le j \\
2j-\kappa(n,j,i)    &    \lfloor\frac{ij}{n} \rfloor + \lfloor\frac{(n-i-2)j}{n}\rfloor = j-2,  \mbox{and }j \le \kappa(n,j,i) \le 2j-1 \\
0        &  \  \mbox{otherwise} 
\end{array} \right.
\end{displaymath}

\end{proposition}

\begin{proof}
Given $j \in \{2,\ldots, \lfloor \frac{n}{2} \rfloor\}$, we use
Proposition 
\ref{fakh 5.2}  to solve for possible indices
 $i \in \{1,\ldots, \lfloor \frac{n}{2} \rfloor -1\}$ such that
 $F_{1,1,i, n-2-i} \cdot D^n_{1,j} \neq 0$.   As a first step, we explicitly determine $ij \% n = ij -\alpha n$ and $(n-2-i)j \% n  = (n-2-i)j - \beta n$.  Thus
 $(\alpha+1)n > ij \ge \alpha n$ and $(\beta+1)n > (n-2-i)j \ge \beta n$.  Since $i \le n-2-i$, one has that $(\beta + 1)n >(n-2-i)j \ge ij \ge \alpha n$ and in particular $\alpha \le \beta$. We get a nonzero intersection if and only if $(n-2-i)j \% n + ij \% n + 2j = 2n$, that is, $\alpha + \beta=j-2$.  In particular we have the bound $0 \le \alpha \le \lfloor \frac{j-2}{2} \rfloor$. 

Now substitute $j-2-\alpha$ for $\beta$, giving $(\alpha + 1)n + (n-2j) > ij \ge \alpha n + (n-2j)$.  Since $n-2j \ge 0$, we get the tighter bound $(\alpha +1)n > ij \ge \alpha n + (n-2j)$. Now taking $\kappa = \kappa(n,j,i)$ as above, $ij = \alpha n + n-\kappa$ and substituting this into our tighter inequality gives $\kappa\le 2j$.
We can now compute the intersection numbers.    Using Fakhruddin's notation we let $\nu_m$ for $1 \le m \le 4$ denote the integers representing the $\bmod n$ congruence classes of $j,j,ij$ and $(n-2-i)j$ respectively. Since $0<2j \le n$ we can take $\nu_1 = \nu_2 = j$ and by definition $\nu_3 = ij \%n  = n-\kappa$. When the intersection number is nonzero, this forces $\nu_4= (n- 2-i)j \% n$ to be $n+\kappa-2j$.   

In the case $\kappa \le j$, one has that  $n-\kappa-j \ge n-2j \ge 0$ and so $\nu_3 \ge \nu_1=\nu_2$. Also $n-\kappa-(n+\kappa-2j)=2(j-\kappa) \ge 0$, so $\nu_3 \ge \nu_4$.  This gives that $\nu_3$ is the maximum of the $\nu_i$.   If $\nu_4$ is the minimum, then $\min + \max= n+(n-2j) \ge n$.  If $\nu_1=\nu_2=j$ is the minimum, then $\min+\max=j+n-\kappa=n+(j-\kappa) \ge n$.  In either case, the intersection number is $n-\nu_3=\kappa$, as asserted.  In case $\kappa \ge j$, one has that $\nu_4$ is maximum of the $\nu_i$. 
Indeed, $\nu_4 \ge \nu_3$ since $(n+\kappa-2j)-(n-\kappa)=2(\kappa-j) \ge 0$, and $\nu_4 \ge \nu_2=\nu_1$, since $n+\kappa-2j-j\ge n-2j \ge 0$.  If $\nu_3$ is the minimum of the $\nu_i$,
then $\max+\min=n+\kappa-2j+n-\kappa=n+(n-2j)\ge n$.  If $\nu_1=\nu_2=j$ is the minimum of the $\nu_i$,
then $\max+\min=n+\kappa-2j+j\ge n$.  In either case the intersection number is $n-\nu_4=n-(n+\kappa-2j)=2j-\kappa$, as asserted.
\end{proof}

\section{
  \texorpdfstring{Extremality of the divisors $D^{n}_{1,j}$ in case $1
    \leq j \leq 4$ or $r=0$}{Extremality of the divisors Dn1j in case
    1<=j<=4 or r=0}} \label{special cases section}

We split the proof of the Extremality Theorem into four cases.
Throughout the sequel we write $n=jk+r$ where $0 \leq r <j$.  First, we consider the special cases
$j=1, 2,3,4$.  Second,
we consider the case $r=0$.  Third, we will consider the case $k\geq 3$ and $r>0$.
Fourth, we will consider the case $k=2$ and $r>0$.  In each case, we
will define a family $\Ghat$ consisting of $g-1$ curves.  

In this section we will define a family of curves $\Ghat$ in the special cases 
$1 \leq j \leq 4$ or $r=0$.  We consider the matrix $C$ given by the coefficients when the curves in $\Ghat$ are written in the basis $\{F_{q,1,1}\}$.  (We will use the letters $p,q$ to index rows and columns of matrices.)  

The matrix $C$ has $g-1$ rows and $g$ columns.  We will argue that $C$ has full rank by exhibiting a $(g-1) \times (g-1)$ minor $\widehat{C}$ with nonzero determinant. Our choice of the minor corresponds to the curve that is dropped in the construction of the family $\Ghat$.

\subsection{If \texorpdfstring{$1 \leq j \leq 4$}{1<=j<=4}}

We consider the cases $j=1,2,3,4$.  Note that if $k=2$ and $j\leq 4$, then $n\leq 11$, and one can check the extremality of $D^{n}_{1,j}$ in these cases by direct calculations.  So in the arguments below, we may assume that $k \geq 3$.  

\subsubsection{\texorpdfstring{$j=1$}{j=1}}  Suppose $j=1$.  Then we can show that $D^n_{1,1}$ is trivial.  Indeed, we may use Propositions \ref{fakh 5.2} and \ref{when int numbers are zero} to show that $D^n_{1,1} \cdot F_{i,1,1} =0$ for all $\F$-curves of the form $F_{i,1,1}$ for $i \le g$.   These curves form a basis for the vector space of $1$-cycles modulo numerical equivalence, by Proposition \ref{C1}.

\subsubsection{\texorpdfstring{$j=2$}{j=2}}  Suppose that $j=2$.  We may also assume $k \geq 3$ by the note above.   We use Propositions \ref{fakh 5.2} and \ref{when int numbers are zero} d) to show that $D^{n}_{1,2} \cdot F_{i,1,1} = 0$ if $1 \leq i \leq g-1$.  Thus we take $\Ghat = \{F_{i,1,1} : 1 \leq i \leq g-1 \}$.  This is a set of $g-1$ distinct curves; they are independent by Proposition \ref{C1}; and hence the $D^{n}_{1,2}$ is extremal.

\subsubsection{\texorpdfstring{$j=3$}{j=3}}  Suppose $j=3$.  We may also assume $k \geq 3$ by the note above.   If $r=0$ then we may use Propositions \ref{fakh 5.2} and \ref{when int numbers are zero} d) to show that $D^{n}_{1,3} \cdot F_{i,1,1} = 0$ if $i  \in \{1,\ldots,g\} \smallsetminus \{k-1\}$.  We take $\Ghat$ to be this family. This is a set of $g-1$ distinct curves; they are independent by Proposition \ref{C1}; and hence the $D^{n}_{1,3}$ is extremal in this case.

If $r=1,2$ then we may use Propositions \ref{fakh 5.2} and \ref{when int numbers are zero} d) to show that $D^{n}_{1,3} \cdot F_{i,1,1} = 0$ if $i  \in \{1,\ldots,g\} \smallsetminus \{k-1,k\}$, and also that $D^{n}_{1,3} \cdot F_{k-1,k,k} = 0$.  We take $\Ghat = \{F_{i,1,1} : 1 \leq i \leq g, i \neq k-1 \} \cup \{ F_{k-1,k,k} \}$.

Next, we consider the matrix $C$ described above.  We compute $c_{k-1,k-1} $ using Proposition \ref{gamma j} and find that this is $-r \neq 0$.  Hence $C$ has full rank, the family $\Ghat$ is independent, and $D^n_{1,3}$ is extremal.

\subsubsection{\texorpdfstring{$j=4$}{j=4}}   Suppose $j=4$.  We may also assume $k \geq 3$ by the note above.  We use Propositions \ref{fakh 5.2} and \ref{when int numbers are zero} d) to show that $D^{n}_{1,j} \cdot F_{i,1,1} = 0$ unless $i$ is in the set $I$ below.  We form a family $\Ghat$ in each case as follows:
\begin{center}
\begin{tabular}{lll}
$r$     & $I$ & $\Ghat$ \\ 
$0$ & $\{k-1,2k-1=g \}$ & $\{F_{i,1,1} : 1 \leq i \leq g-1, i \neq k-1 \} \cup \{ F_{k-1,k,k} \}$\\ 
$1$ & $\{k-1,k,2k-1=g \}$ & $\{F_{i,1,1} : 1 \leq i \leq g-1, i \neq k-1,k \} \cup \{ F_{k-1,k,k},  F_{k,k,k}\}$\\ 
$2$ & $\{k-1,k,2k=g \}$ & $\{F_{i,1,1} : 1 \leq i \leq g-1, i \neq k-1,k \} \cup \{ F_{k-1,k,k}, F_{k,k,k}\}$\\ 
$3$ & $\{k-1,k,2k=g \}$ &$\{F_{i,1,1} : 1 \leq i \leq g-1, i \neq k-1,k \} \cup \{ F_{k-1,k,k}, F_{k,k,k}\}$
\end{tabular}
\end{center}
If $r=0$ the only $k,k$-curve in $\Ghat$ is $F_{k-1,k,k}$.  Thus most rows of the matrix $C$ look like rows of the identity matrix.  We compute $c_{k-1,k-1} = -k+1 \neq 0$.  Thus $C$ has full rank, and hence $\Ghat$ is independent and $D^{n}_{1,4}$ is extremal.  

If $r>0$ then we only need to check the following $2 \times 2$ minor of $C$:
\begin{displaymath}
\left(
\begin{array}{cc}
c_{k-1,k-1} & c_{k-1,k} \\
c_{k,k-1} & c_{k,k}  
\end{array} \right) = 
\left(
\begin{array}{cc}
-k+1 & -k+3 \\
-k & -k+2 
\end{array} \right).
\end{displaymath}
This has determinant $2$.  Thus $C$ has full rank, $\Ghat$ is independent, and $D^{n}_{1,4}$ is extremal.

\subsection{If \texorpdfstring{$r=0$}{r=0}} \label{r=0 section}
When $r=0$ we may use Proposition \ref{fakh 5.2} to show that $D^{n}_{1,j} \cdot F_{i,1,1} = 0$ if $i \not\equiv -1 \bmod k$, and that $D^{n}_{1,j} \cdot F_{i,k,k} = 0$ if $i \equiv -1 \bmod k$.  This motivates us to define $\mathcal{G} = \{ F_{i,1,1} : 1 \leq i \leq g, i \not\equiv -1 \bmod k \} \cup \{ F_{i,k,k} : 1 \leq i \leq g, i \equiv -1 \bmod k \} $.  We then throw out the last $k,k$-curve to get a family $\Ghat$.  


Suppose $j \geq 5$ (the cases $j=2,3,4$ were covered in the previous subsection).  We form a matrix $N$ as follows: 
As above, start with the matrix $C$ given by the coefficients when the curves in $\mathcal{G}$ are written in the basis $\{F_{q,1,1}\}$.  We delete the row and column corresponding to the curve that was dropped in the construction of the family $\Ghat$ to obtain a $(g-1) \times (g-1)$ minor $\widehat{C}$.

To show that $\widehat{C}$ is full rank, it is enough to show that the minor consisting of rows which come from $k,k$-curves, and the corresponding columns, has nonzero determinant.  Write $m=\lfloor \frac{j}{2} \rfloor -1$.  Then we extract this $m \times m$ minor and call it $M$.   (See Example \ref{25-7} below.)   Let $N$ be the matrix obtained from $M$ by subtracting row $m$ from rows 1 through $m-1$. 

Then $N$ is the following $m \times m$ matrix:
\begin{displaymath}
\left( \begin{array}{ccccccccc}
1 & k-1 & 0 & 0 & \cdots & 0& -1 \\
0 & 1 & k-1 & 0 & \cdots & 0& -1 \\
0 & 0 & 1& k-1 & \ddots & \vdots& -1 \\
0 & 0 & \ddots & \ddots & \ddots &  0& -1 \\
0 & 0 & 0 & 0 &  1&  k-1& -1 \\
0 & 0 & 0 & 0 & 0 &  1& k-2 \\
-k & 0 & 0 & 0 & 0 &  0& 1 \\
\end{array} \right)
\end{displaymath}

We may row reduce across the bottom row to compute the determinant of the matrix, and find $\det N = (-k+1)^{m}$.  Thus, $\det \widehat{C} \neq 0$, the family $\Ghat$ is independent, and $D^{n}_{1,j}$ is extremal when $r=0$.

\section{\texorpdfstring{The family $\mathcal{G}$ when $k \geq 3$ and
    $r>0$}{The family G when k>=3 and r>0}} \label{k geq 3 structure}
Given $n$ and $j$, as before we let $k:= \lfloor n/j \rfloor$ and $r:=n-jk$.  Thus, we have $n=jk+r$, where $k$ is a positive integer, and $0 < r < j$.  (The case $r=0$ was handled in Section \ref{r=0 section}.)

In this section, for each $D^{n}_{1,j}$ with $k \geq 3$,  we will
describe a family of curves $\mathcal{G} = \mathcal{G}^{n}_{1,j}$
which will be used in the next section to prove the extremality of the divisor $D^{n}_{1,j}$.  (See Sections \ref{k=2 structure} and \ref{k=2 independence} for the case $k=2$.)

Suppose integers $n$ and $j$ are given.  Recall that $g$ is defined by $n=2g+2$ or $n=2g+3$, and suppose $2 \le j \le g+1$.  

We define 
\begin{displaymath} T_{j,n} := \left\{ 
\begin{array}{ll} \{ t \frac{n}{\gcd(j,r)} : t \in \{ 1,\ldots, \left\lfloor \frac{\gcd(j,r)}{2} \right\rfloor \} \}  & \mbox{if $r\geq 2$}\\
\emptyset & \mbox{if $r=1$}
\end{array} \right.
\end{displaymath}

\begin{definition}[Definition of $\Ghat$ when $k \geq 3$] \label{k geq
    3 def of G}\label{k geq 3 family}
Let $n$ and $j$ be as above.  Suppose $k \geq 3$, $j \geq 4$, and
$r>0$.  (For the cases $r=0$ and $j\leq 4$, see 
Section \ref{special cases section}.)



We define a set of $\F$-curves $\mathcal{G}$ as follows:   For each  $i \in \{1, \ldots, g\}$, if $i \in T_{j,n}$, then put $\mathcal{G}(i) = F_{i,k,k}$.  If $i \not\in T_{j,n}$ and $D^n_{1,j} \cdot F_{i,1,1} = 0$, then put $\mathcal{G}(i)=F_{i,1,1}$.  Otherwise, if $i\not\in T_{j,n}$ and $D^n_{1,j} \cdot F_{i,1,1} \ne 0$, then put  $\mathcal{G}(i)=F_{i,k,k}$. 
Define $\mathcal{G} =\{ \mathcal{G}(i):  1 \leq i \leq g \}$.  Let $p$ be the largest index $i$ such that $\mathcal{G}(i) = F_{i,k,k}$.  Then we also define a family $\Ghat$ of $g-1$ curves as $\Ghat := \mathcal{G} \smallsetminus \{ F_{p,k,k} \}$.   
\end{definition}

\begin{example}
Let $n=20$ and $j=6$.  We abbreviate and write $D_{6}$ for $D^{20}_{1,6}$.  Then $g=9$, so to prove that $D_{6}$ is extremal, we need $9$ independent $\F$-curves which intersect $D_{6}$ in degree zero.  Also $k = \lfloor 20/6 \rfloor = 3$, so if necessary we will replace curves of the form $F_{1,1,m}$ with curves of the form $F_{3,3,m}$.  

We compute intersection numbers and choose curves below:
\begin{displaymath}
\begin{array}{ll}
\underline{F_{m,1,1}  \cdot D_{6}}& \underline{\mbox{Replace by }}\\
\fcolorbox{red}{white}{$F_{1,1,1}  \cdot D_{6} = 0 $}& \\ 
F_{2,1,1} \cdot D_{6} = 4 & \fcolorbox{red}{white}{$F_{2,3,3} \cdot D_{6} = 0$}   \\ 
F_{3,1,1}  \cdot D_{6} = 2  & \fcolorbox{red}{white}{$F_{3,3,3} \cdot D_{6} = 0$}  \\ 
\fcolorbox{red}{white}{$F_{4,1,1} \cdot D_{6} =0 $}& \\ 
F_{5,1,1} \cdot D_{6} = 4 & \fcolorbox{red}{white}{$F_{5,3,3} \cdot D_{6} = 0 $} \\ 
F_{6,1,1} \cdot D_{6} = 2 & \fcolorbox{red}{white}{$F_{6,3,3} \cdot D_{6} = 0 $} \\ 
\fcolorbox{red}{white}{$F_{7,1,1} \cdot D_{6} = 0 $}& \\ 
\fcolorbox{red}{white}{$F_{8,1,1} \cdot D_{6} = 0 $}\\ 
F_{9,1,1} \cdot D_{6} = 6 &  \fcolorbox{red}{white}{$F_{9,3,3} \cdot D_{6} = 0  $} 
\end{array}
\end{displaymath}

Thus the family $\mathcal{G} = \{ F_{1,1,1}, F_{2,3,3}, F_{3,3,3}, F_{4,1,1}, F_{5,3,3}, F_{6,3,3}, F_{7,1,1}, F_{8,1,1}, F_{9,3,3} \} $.  Then we drop $F_{9,3,3}$ to get $\Ghat= \{ F_{1,1,1}, F_{2,3,3}, F_{3,3,3}, F_{4,1,1}, F_{5,3,3}, F_{6,3,3}, F_{7,1,1}, F_{8,1,1} \}$.
\end{example}

In a series of propositions below, we will elucidate the structure and properties of the family $\mathcal{G}$ in more detail.

\begin{proposition}  \label{k geq 3 kills curves}  Suppose that $n = jk+r$ with $k\geq 3$.  Then every member of $\mathcal{G}$ intersects $D^{n}_{1,j}$ in degree zero.
\end{proposition}
\begin{proof}  
We need to show two things:
\begin{enumerate}
\item If $i \in T_{j,n}$, then $D^n_{1,j} \cdot F_{i,k,k} = 0$.
\item If $i \not\in T_{j,n}$ and $D^n_{1,j} \cdot F_{i,1,1} \ne 0$, then $D^n_{1,j} \cdot F_{i,k,k} = 0$.
\end{enumerate}

For the first statement: Notice that if $i \in T_{j,n}$, then $ij$ is
an integer multiple of $n$.  Thus $\nu_1=0$, and by Proposition \ref{when int numbers are zero}, this implies that $D^n_{1,j} \cdot F_{i,k,k} = 0$.

For the second statement, let $1 \leq i \leq g$, and suppose $D^{n}_{1,j} \cdot F_{i,1,1} \neq 0$.  Then we need to show that $D^{n}_{1,j} \cdot F_{i,k,k} = 0$.

As usual, we write $n=jk+r$, where $0 \leq r < j$.

By Lemma \ref{fakh 5.2}, we know 
\begin{equation} \label{i11 eq}  (n-i-2)j \% n + ij \% n + 2j = 2n.
\end{equation}
We will show that
\begin{equation}  (n-i-2k)j \% n + ij \% n + 2kj \neq 2n,
\end{equation}
which by  Lemma \ref{fakh 5.2} implies the desired result.

So suppose for purposes of contradiction that 
\begin{equation} \label{ikk eq}  (n-i-2k)j \% n + ij \% n + 2kj = 2n.
\end{equation}
Subtracting (\ref{ikk eq}) from (\ref{i11 eq}) yields
\[  (n-i-2)j \% n - (n-i-2k)j \% n = 2jk-2j
\]
which simplifies to
\begin{equation} \label{to contradict}  (-(i+2)j) \% n  - ( -(i+2k)j)\% n = 2n - 2r - 2j.
\end{equation}
We now study the terms $(-(i+2)j) \% n $ and $ ( -(i+2k)j)\% n$.  First, consider $(-(i+2)j) \% n $.  Write $ij = Qn + R$ where $0 \leq R < n$.  Then 
\begin{equation}
(-(i+2)j) \% n = \left\{ 
\begin{array}{ll}   -R-2j+n &\mbox{if $n \geq R+2j$}\\
-R-2j+2n & \mbox{if $n < R+2j$}
\end{array}  
\right.
\end{equation}

Next, consider $ ( -(i+2k)j)\% n$.  Since $jk=n-r$ this is equal to $ ( -ij+2r)\% n$, so 
\begin{equation}
(-(i+2k)j) \% n = \left\{ 
\begin{array}{ll}   -R+2r &\mbox{if $2r \geq R$}\\
-R+2r+n & \mbox{if $2r < R$}
\end{array}  
\right.
\end{equation}

We now study each of the four cases that arise from the two formulas above.

Case 1: $n \geq R+2j$, $2r \geq R$. Then (\ref{to contradict}) reads $-R-2j+n  + R - 2r = 2n-2r -2j$, a contradiction.  

Case 2: $n \geq R+2j$, $2r < R$.  Then (\ref{to contradict}) reads $-R-2j+n + R -2r -n = 2n - 2r -2j$, a contradiction.  

Case 3: $n < R+2j$, $2r \geq R$.  Since $r<j$, these inequalities imply $n < 2j+2r < 3j+r$, which contradicts the hypothesis that $n=jk+r$ with $k \geq 3$.

Case 4: $n < R+2j$, $2r < R$.  Then (\ref{to contradict}) reads $-R-2j+2n  + R-2r -n = 2n - 2r -2j$, a contradiction.  
\end{proof}

\begin{proposition} \label{k geq 3 beginnings}
Suppose $k \geq 3$ and $r >0$.  Suppose $j \geq 5$.
The family $\mathcal{G}$ begins as follows: 
\begin{enumerate}
\item $\mathcal{G}(i) = F_{i,1,1}$ for $i =1,\ldots, k-2$; 
\item $\mathcal{G}(k-1) = F_{k-1,k,k}$; 
\item $\mathcal{G}(k) = F_{k,k,k}$.
\end{enumerate}
\end{proposition}
\begin{proof} 
We use Proposition \ref{when int numbers are zero}.  Since $0 < r < j$, we have $jk< jk+r < j(k+1)$, and hence $\frac{1}{k+1} < \frac{j}{n} < \frac{1}{k}$.   Then for $i \leq k-2$, $0 <\frac{(i+2)j}{n}<1$, so $i \not\in T_{j,n}$.  There is no $p\in\Z$ such that $\frac{ij}{n}p <\frac{(i+2)j}{n}$, so $D^n_{1,j} \cdot F_{i,1,1} = 0$.  Hence  $\mathcal{G}(i) = F_{i,1,1}$ for $i =1,\ldots, k-2$.  

If $i=k-1$, then $\frac{ij}{n} < 1 < \frac{(i+2)j}{n} $, so $D^n_{1,j} \cdot F_{k-1,1,1} \neq 0$ and  $\mathcal{G}(k-1) = F_{k-1,1,1}$.  

Similarly, for $i=k$, we have $\frac{ij}{n} < 1 < \frac{(i+2)j}{n}$,  so $D^n_{1,j} \cdot F_{k,1,1} \neq 0$ and $\mathcal{G}(k) = F_{k,1,1}$.





\end{proof}

\begin{proposition} \label{k geq 3 pairs}
Suppose $k \geq 3$ and $r >0$.  Suppose $j \geq 5$.

The $k,k$ curves in the family $\mathcal{G}$ show up in the following patterns:
\begin{enumerate} 
\item a pair of $k,k$-curves, followed by $(k-1)$ $1,1$-curves, followed by a pair
of $k,k$-curves.  We call the first such pair in this sequence a \emph{$(k-1)$ pair}.
\item a pair of $k,k$-curves, followed by $(k-2)$ $1,1$-curves, followed by a pair
of $k,k$-curves.  We call the first such pair in this sequence a \emph{$(k-2)$ pair}.
\end{enumerate}
\end{proposition}
\begin{proof}
We proceed inductively.  For the base case, by Proposition \ref{k geq 3 beginnings} we know the first $k$ curves in $\mathcal{G}$ are   
\begin{eqnarray*}
\mathcal{G}(i) & = & F_{i,1,1} \mbox{ for $i=1,\ldots,k-2$}\\
\mathcal{G}(k-1) & = & F_{k-1,k,k} \\
\mathcal{G}(k) & = & F_{k,k,k} 
\end{eqnarray*}

For the induction step, we will show that if 
\begin{eqnarray*}
\mathcal{G}(i-1) & = & F_{i-1,1,1} \\
\mathcal{G}(i) & = & F_{i,k,k} \\
\mathcal{G}(i+1) & = & F_{i+1,k,k} 
\end{eqnarray*}
then $F_{i,k,k}$ and $F_{i+1,k,k}$ are followed by either $(k-2)$ or $(k-1)$ $1,1$-curves, followed by another pair of $k,k$-curves.


Note that since $0 < r < j$, we have  $\frac{1}{k+1} < \frac{j}{n} < \frac{1}{k}$.  

\textit{Case 1.}  Suppose that $i,i+1 \not\in T_{j,n}$.  Then we know $D^{n}_{1,j} \cdot F_{i,1,1} \neq 0$ and $D^{n}_{1,j} \cdot F_{i+1,1,1} \neq 0$.  By Proposition \ref{when int numbers are zero}, 
 there exists $p\in\Z$ such that $\frac{ij}{n} < p < \frac{(i+2)j}{n}$ and  $q$ such that $\frac{(i+1)j}{n} < q < \frac{(i+3)j}{n}$.  Combining these statements with our bounds on $j/n$ and $k$, we conclude there exists $p\in\Z$ such that $\frac{(i+1)j}{n} < p < \frac{(i+2)j}{n}$.

Next we argue that $\frac{(i+k+1)j}{n} \leq p+1$.  Suppose not.  Then $p+1 < \frac{(i+1)j}{n} + k \frac{j}{n} < p+1$, a contradiction.  

Next we argue that $p+1< \frac{(i+k+3)j}{n}$.  We know $p< \frac{(i+2)j}{n} $ and $\frac{1}{k+1} < \frac{j}{n} $, so $p+1 < \frac{(i+2)j}{n} + (k+1)\frac{j}{n} = \frac{(i+k+3)j}{n}$.

Then there are three possibilities.  If $\frac{(i+k+2)j}{n} > p+1$, then we have $\frac{(i+k)j}{n} < \frac{(i+k+1)j}{n} < p+1 < \frac{(i+k+2)j}{n} < \frac{(i+k+3)j}{n} $.  It follows that $D^{n}_{1,j} \cdot F_{i+z,1,1} =0$ for $2,\ldots,k-1$, and moreover no such $i+z \in T_{j,n}$.  We also have  $D^{n}_{1,j} \cdot F_{i+k,1,1} \neq 0$ and $D^{n}_{1,j} \cdot F_{i+k+1,1,1} \neq 0$.  Thus
\begin{eqnarray*}
\mathcal{G}(i+z) & = & F_{i+z,1,1} \mbox{ for $z=2,\ldots,k-1$}\\
\mathcal{G}(i+k) & = & F_{i+k,k,k} \\
\mathcal{G}(i+k+1) & = & F_{i+k+1,k,k} 
\end{eqnarray*}
Thus, the pair $F_{i,k,k}$ and $F_{i+1,k,k}$ is a $k-2$ pair.

If $\frac{(i+k+2)j}{n} = p+1$,  then we have $\frac{(i+k+1)j}{n} < \frac{(i+k+2)j}{n} = p+1< \frac{(i+k+3)j}{n} < \frac{(i+k+4)j}{n} $.  It follows that $D^{n}_{1,j} \cdot F_{i+z,1,1} =0$ for $2,\ldots,k$, and moreover no such $i+z \in T_{j,n}$.  We have  $D^{n}_{1,j} \cdot F_{i+k+1,1,1} \neq 0$.  We have $i+k+2 \in T_{j,n}$.   Thus
\begin{eqnarray*}
\mathcal{G}(i+z) & = & F_{i+z,1,1} \mbox{ for $z=2,\ldots,k$}\\
\mathcal{G}(i+k+1) & = & F_{i+k+1,k,k} \\
\mathcal{G}(i+k+2) & = & F_{i+k+2,k,k} 
\end{eqnarray*}
Thus, the pair $F_{i,k,k}$ and $F_{i+1,k,k}$ is a $k-1$ pair.

If $\frac{(i+k+2)j}{n} < p+1$, then we have $\frac{(i+k+1)j}{n} < \frac{(i+k+2)j}{n} <p+1< \frac{(i+k+3)j}{n} < \frac{(i+k+4)j}{n} $.  It follows that $D^{n}_{1,j} \cdot F_{i+z,1,1} =0$ for $2,\ldots,k$, and moreover no such $i+z \in T_{j,n}$.  We also have  $D^{n}_{1,j} \cdot F_{i+k+1,1,1} \neq 0$ and $D^{n}_{1,j} \cdot F_{i+k+2,1,1} \neq 0$.  Thus
\begin{eqnarray*}
\mathcal{G}(i+z) & = & F_{i+z,1,1} \mbox{ for $z=2,\ldots,k$}\\
\mathcal{G}(i+k+1) & = & F_{i+k+1,k,k} \\
\mathcal{G}(i+k+2) & = & F_{i+k+2,k,k} 
\end{eqnarray*}
Thus, the pair $F_{i,k,k}$ and $F_{i+1,k,k}$ is a $k-1$ pair.

\textit{Case 2.}  Suppose that $i \in T_{j,n}$.  We claim this cannot happen.  Since $i \in T_{j,n}$, we have $\frac{ij}{n} = p \in \Z$.  But then we have $i-1,i+1 \not\in T_{j,n}$, and there exists $p \in Z$ such that $\frac{(i-1)j}{n} < p < \frac{(i+1)j}{n}$.  Then $D^{n}_{1,j} \cdot F_{i-1,1,1}\neq 0$, but this contradicts the induction hypothesis.  

\textit{Case 3.}  Suppose that $i+1 \in T_{j,n}$.  Then we know $\frac{(i+1)j}{n} = p \in \Z$.  We can show that $\frac{(i+k+1)j}{n} < p+1 < \frac{(i+k+2)j}{n} $, and thus 
\begin{eqnarray*}
\mathcal{G}(i+z) & = & F_{i+z,1,1} \mbox{ for $z=2,\ldots,k-1$}\\
\mathcal{G}(i+k) & = & F_{i+k,k,k} \\
\mathcal{G}(i+k+1) & = & F_{i+k+1,k,k} 
\end{eqnarray*}
Thus, the pair  $F_{i,k,k}$ and $F_{i+1,k,k}$ is a $k-2$ pair.
\end{proof}

\begin{definition} \label{k geq 3 s}
Suppose $k \geq 3$ and $r >0$. 
We define the \emph{sequence of pairs $(s_i)$ associated to $\mathcal{G}$} as follows:  Write $L$ for the last $k,k$-curve in $\Ghat$.  
Group the $k,k$-curves $F_{i,k,k}$ in $\mathcal{G}$ with $i<L$ according to whether they fall into $(k-1)$ pairs or $(k-2)$ pairs as defined by the Proposition above.  
Define $s_{i}=k-1$ if the $i^{th}$ group is a $(k-1)$ pair, and $s_{i}=k-2$ if the $i^{th}$ group is a $(k-2)$ pair.  Note: we start indexing $(s_i)$ with $i=0$.
\end{definition}

In the following proposition we study how the family $\mathcal{G}$ ends.

\begin{proposition} \label{k geq 3 endings}  Write $n=jk+r$ with $0 \leq r < j$.  Suppose $k \geq 3$, $j \geq 6$ and $r>0$.   
\begin{enumerate}
\item If $j$ is even: the curve $F_{g,1,1}$ is dropped, and it is preceded either by a $k-1$ pair or a $k-2$ pair, so $F_{L,k,k}$ is either $F_{g-k+1,k,k}$ or $F_{g-k-1,k,k}$.
\item If $j$ is odd: the curve that is dropped is always second in a pair, and is given as follows:
\begin{enumerate}
\item If $n$ is even, $j$ is odd, $k$ is even: $F_{g-\frac{k}{2},k,k}$ is dropped, and $F_{L,k,k} =F_{g-\frac{k}{2}-1,k,k}$
\item If $n$ is even, $j$ is odd, $k$ is odd: $F_{g-\frac{k-1}{2},k,k}$ is dropped, and $F_{L,k,k} =F_{g-\frac{k-1}{2}-1,k,k}$.
\item If $n$ is odd, $j$ is odd, $k$ is even: $F_{g-\frac{k}{2}+1,k,k}$ is dropped, and $F_{L,k,k} = F_{g-\frac{k}{2},k,k}$.
\item If $n$ is odd, $j$ is odd, $k$ is odd: $F_{g-\frac{k+1}{2}+1,k,k}$ is dropped, and $F_{L,k,k} =F_{g-\frac{k+1}{2},k,k}$.
\end{enumerate}
\end{enumerate}
\end{proposition}
\begin{proof}
\textit{Case 1: j even.}  Suppose $j$ is even.  We show that $D^{n}_{1,j} \cdot F_{g,1,1} \neq 0$ and $D^{n}_{1,j} \cdot F_{g-1,1,1} =0$.  Then by Proposition \ref{k geq 3 pairs} and checking that $g,g-1 \not\in T_{j,n}$, the result follows.

   Suppose $n$ is even.  Then to compute $D^{n}_{1,j} \cdot F_{g,1,1,g} $ using Prop. \ref{fakh 5.2} we have $\nu_{1}=\nu_{4} = n-j \neq 0$ and $\nu_2 = \nu_3=j \neq 0$.  By Prop. \ref{when int numbers are zero} c) we have   $D^{n}_{1,j} \cdot F_{g,1,1,g} \neq 0$, and $\mathcal{G}(g)=F_{g,k,k}$.  As this is the last possible curve in $\mathcal{G}$, it is the last $k,k$-curve in $\mathcal{G}$, and hence the curve that is thrown out.  

Next we compute $D^{n}_{1,j} \cdot F_{g-1,1,1,g+1} $.  We see that $\nu_{4}=0$ and hence by Prop. \ref{when int numbers are zero} b) we have $D^{n}_{1,j} \cdot F_{g-1,1,1,g+1} =0$.  Since $g-1 \not\in T_{j,n}$ we have $\mathcal{G}(g-1)=F_{g-1,k,k}$, and the result follows. 

Suppose $n$ is odd.  We compute $D^{n}_{1,j} \cdot F_{g,1,1,g+1} $.  Write $j=2m$.  We see that $\nu_{1}=n-3m$, $\nu_2=\nu_3=j$, $\nu_4=n-m$.  We have $\nu_i \neq 0$ for all $i$ and $\sum \nu_i = 2n$. Hence by Prop. \ref{when int numbers are zero} c) we have $D^{n}_{1,j} \cdot F_{g,1,1,g+1} \neq 0$.  Since $g \not\in T_{j,n}$ we have $\mathcal{G}(g)=F_{g,k,k}$.  As this is the last possible curve in $\mathcal{G}$, it is the last $k,k$-curve in $\mathcal{G}$, and hence the curve that is thrown out.  

Next we compute $D^{n}_{1,j} \cdot F_{g-1,1,1,g+2}$. Write $j=2m$.  We see that $\nu_{1}=n-5m$, $\nu_2=\nu_3=j$, $\nu_4=m$.  We have $\sum \nu_i \neq  2n$. Hence by Prop. \ref{when int numbers are zero} a) we have $D^{n}_{1,j} \cdot F_{g,1,1,g+1} = 0$.  Since $g-1 \not\in T_{j,n}$ we have $\mathcal{G}(g-1)=F_{g-1,k,k}$, and the result follows.

\textit{Case 2: j odd.}  It is enough to show in each case that $D^{n}_{1,j} \cdot F_{i,1,1} \neq 0$ and $D^{n}_{1,j} \cdot F_{i-1,1,1} \neq 0$ (where $i$ is given in the statement of the theorem).  Then regardless of whether $i, i-1 \in T_{j,n}$ we get $\mathcal{G}(i)=F_{i,k,k}$ and $\mathcal{G}(i-1)=F_{i-1,k,k}$.  Then we can use Prop. \ref{k geq 3 endings} to see that $F_{i,k,k}$ is the last $k,k$-curve in $\mathcal{G}$ and hence is the curve that is dropped.

\textit{Case 2a: $n$ is even, $j$ is odd, $k$ even.}

We compute $D^{n}_{1,j} \cdot F_{g-\frac{k}{2},1,1,g+\frac{k}{2}}$.  
 We have $\nu_1 = n-j+r/2$, $\nu_2=\nu_3=j$,$\nu_{4} = j(k-1)+r/2$.
 We have $\nu_i \neq 0$ for all $i$ and $\sum \nu_i = 2n$. Hence by
 Prop. \ref{when int numbers are zero} c) we have $D^{n}_{1,j} \cdot
 F_{g-\frac{k}{2},1,1,g+\frac{k}{2}} \neq 0$.  

Next, we compute $D^{n}_{1,j} \cdot
F_{g-\frac{k}{2}-1,1,1,g+\frac{k}{2}+1}$.  We have $\nu_1 = n-2j+r/2$,
, $\nu_2=\nu_3=j$, $\nu_{4} = jk+r/2$.
 We have $\nu_i \neq 0$ for all $i$ and $\sum \nu_i = 2n$. Hence by
 Prop. \ref{when int numbers are zero} c) we have $D^{n}_{1,j} \cdot
 F_{g-\frac{k}{2},1,1,g+\frac{k}{2}} \neq 0$.  

Thus we have $\mathcal{G}(g-\frac{k}{2})=F_{g-\frac{k}{2},k,k}$ and
$\mathcal{G}(g-\frac{k}{2}-1)=F_{g-\frac{k}{2}-1,k,k}$.  By
Prop. \ref{k geq 3 endings}, we see that $F_{g-\frac{k}{2},k,k}$ is
the last $k,k$-curve in $\mathcal{G}$, and hence it is the curve that
is dropped in forming $\Ghat$.  

\textit{Case 2b, 2c, and 2d.}  The proofs for the remaining cases are 
similar to the proof of Case 2a above.  We will simply report the
$\nu_i$ that are needed.  

\begin{center}
\begin{tabular}{llll}
Case & Curve & $\nu_1$ & $\nu_4$\\
2b & $F_{g-\frac{k-1}{2},k,k}$ & $n-j+r/2$ &$ j(k-1)+r/2$ \\
2b & $F_{g-\frac{k}{2}-1,1,1,g+\frac{k}{2}+1}$ & $n-2j+r/2$ & $jk+r/2$\\
2c & $F_{g-\frac{k}{2}+1,k,k}$ & $(2n+r-j)/2$ & $(2n-r-3j)/2$ \\
2c & $F_{g-\frac{k}{2},k,k}$ & $(2n+r-3j)/2$ & $(2n-r-j)/2$ \\
2d & $F_{g-\frac{k+1}{2}+1,k,k}$ & $(2n+r-2j)/2$ & $(2n-r-2j)/2$ \\
2d & $F_{g-\frac{k+1}{2},k,k}$ & $(2n+r-4j)/2$ & $(2n-r)/2$ 
\end{tabular}
\end{center}
\end{proof}

\begin{proposition} \label{k geq 3 distinct}
Suppose $k \geq 3$.  The family $\Ghat$ consists of $g-1$ distinct curves.
\end{proposition}
\begin{proof}
The concern is that in the definition of the family $\mathcal{G}$, we
might pick the same curve twice, i.e. $\mathcal{G}(i) =
\mathcal{G}(i')$ where $i \neq i'$.  

This will never happen for two $1,1$-curves.  

If $k \geq 3$, then since the non-$1,1$-curves are $k,k$-curves,
duplication implies that we have $F_{i,k,k} = F_{i',k,k}$.  Thus $i' =
n-i-2k$ and $i' \leq g$, and this could only happen in the range $i
\geq g-2k+2$.  By Propositions \ref{k geq 3 pairs} and \ref{k geq 3 endings}, we know exactly which $k,k$-curves appear in this range, and we can check case-by-case that, after dropping the last $k,k$-curve in forming $\Ghat$, there are never any duplicate curves.

\end{proof}

\section{\texorpdfstring{Extremality of the divisors $D^{n}_{1,j}$
    when $k \geq 3$}{Extremality of the divisors Dn1j when k>=3}}\label{k geq 3 independence}
Suppose $k \geq 3$, $r>0$, and  $j \geq 6$.

As before, we form a matrix $C$ whose $p,q^{th}$ entry is the coefficient on $F_{q,1,1}$ when the curve $\mathcal{G}(p)$ is written in the basis $\{F_{q,1,1} \}$.  We then drop the column corresponding to the curve which is dropped when creating the family $\Ghat$.  By Proposition \ref{k geq 3 endings} the column which is deleted is as follows:
\begin{enumerate}
\item If $n$ is even and $j$ is even: column $g$. 
\item If $n$ is odd, $j$ is even: column $g$.
\item If $n$ is even, $j$ is odd, $k$ is even: column $g-\frac{k}{2}$.
\item If $n$ is even, $j$ is odd, $k$ is odd: column $g-\frac{k-1}{2}$.
\item If $n$ is odd, $j$ is odd, $k$ is even: column $g-\frac{k}{2}+1$.
\item If $n$ is odd, $j$ is odd, $k$ is odd: column $g-\frac{k+1}{2}+1$.
\end{enumerate}

This yields a $(g-1) \times (g-1)$ minor $\widehat{C}$.  (See Example \ref{25-7} below.)

Rows in $\widehat{C}$ which correspond to $1,1$-curves look like rows of the identity matrix.  Hence, to show that $\widehat{C}$ is full rank, it is enough to show that the minor consisting of rows which come from $k,k$-curves, and the corresponding columns, has nonzero determinant.  We extract this minor and call it $M$.   (See Example \ref{25-7} below.)

Let $m$ be the number of rows/columns in $M$.  We see that the first two columns of $M$ are constant starting in row 3.  Moreover, the lower right $(m-2) \times (m-2)$ block of $M$ is upper triangular with nonzero entries on the diagonal.  

Let $N$ be the matrix obtained from $M$ by subtracting row $m$ from rows 1 through $m-1$. 

\begin{example} \label{25-7} If $j=7$ and $n=25$ we have
\small
\begin{displaymath} C= \left(
\begin{array}{ccccccccccc}
1 & 0 & 0 & 0 & 0 & 0 & 0 & 0 & 0 & 0 & 0\\
-2 & -2 & 0 & 2 & 2 & 1 & 0 & 0 & 0 & 0 & 0\\
-2 & -3 & -1 & 1 & 3 & 2 & 1 & 0 & 0 & 0 & 0\\
0 & 0 & 0 & 1 & 0 & 0 & 0 & 0 & 0 & 0 & 0\\
0 & 0 & 0 & 0 & 1 & 0 & 0 & 0 & 0 & 0 & 0\\
-2 & -3 & -2 & -1 & 0 & 1 & 2 & 3 & 2 & 1 & 0\\
-2 & -3 & -2 & -1 & 0 & 0 & 1 & 2 & 3 & 2 & 1\\
0 & 0 & 0 & 0 & 0 & 0 & 0 & 1 & 0 & 0 & 0\\
-2 & -3 & -2 & -1 & 0 & 0 & 0 & 0 & 1 & 3 & 5\\
0 & 0 & 0 & 0 & 0 & 0 & 0 & 0 & 0 & 0 & 1\\
\end{array} \right),
 \widehat{C}= \left(
\begin{array}{cccccccccc}
1 & 0 & 0 & 0 & 0 & 0 & 0 & 0 & 0 &  0\\
-2 & -2 & 0 & 2 & 2 & 1 & 0 & 0 & 0 &  0\\
-2 & -3 & -1 & 1 & 3 & 2 & 1 & 0 & 0 &  0\\
0 & 0 & 0 & 1 & 0 & 0 & 0 & 0 & 0 &  0\\
0 & 0 & 0 & 0 & 1 & 0 & 0 & 0 & 0 &  0\\
-2 & -3 & -2 & -1 & 0 & 1 & 2 & 3 & 2 &  0\\
-2 & -3 & -2 & -1 & 0 & 0 & 1 & 2 & 3 &  1\\
0 & 0 & 0 & 0 & 0 & 0 & 0 & 1 & 0 &  0\\
-2 & -3 & -2 & -1 & 0 & 0 & 0 & 0 & 1 &  5\\
0 & 0 & 0 & 0 & 0 & 0 & 0 & 0 & 0 &  1\\
\end{array} \right),
\end{displaymath}
\normalsize
\begin{displaymath}
M = \left(
\begin{array}{ccccc}
-2 & 0 & 1 & 0 & 0 \\
-3 & -1 & 2 & 1 & 0 \\
-3 & -2 & 1 & 2 & 2 \\
-3 & -2 & 0 & 1 & 3 \\
-3 & -2 & 0 & 0 & 1 
\end{array}
\right), N = \left(
\begin{array}{ccccc}
1 & 2 & 1 & 0 & -1 \\
0 & 1 & 2 & 1 & -1 \\
0 & 0 & 1 & 2 & 1\\
0 & 0 & 0 & 1 & 2 \\
-3 & -2 & 0 & 0 & 1 
\end{array}
\right)
\end{displaymath}
\end{example} 

The first $m-1$ rows of $N$ are upper triangular with nonzero diagonal entries, by Proposition \ref{cij summary}.

We want to row reduce row $m$ to make $N$ upper triangular, and show that the resulting diagonal entry in row $m$ is also nonzero.  This will show that $N$ (and hence $\widehat{C}$) is full rank.  

Instead of considering the change in $N$ as we row reduce through each column individually, we will only consider the change in $N$ as we row reduce through pairs of columns corresponding to the pairs of  $\Ghat$ recorded in the sequence $(s_i)$.  There are two main reasons for this.  First, the nonzero entries in upper left block of the matrix $N$ are fairly uniform.  As a result, for the ``early'' row reductions, we can compute the change in the diagonal entry simply from knowing whether we are row reducing through a $(k-1)$ pair or a $(k-2)$ pair.  Second, if we only peek at the row reductions in between the blocks, there are always two well-defined nonzero entries in row $m$, which we designate $a_i$ and $b_i$ below.  If we were to consider the row reduction after an arbitrary column (say partway through a pair) there could be 1, 2, or 3 nonzero entries.  Hence, using the sequence of pairs gives us a better framework for recursion.  

\begin{proposition} \label{k geq 3 recurrences}
\begin{enumerate}
\item Suppose two curves $F_{i,k,k}$ and $F_{i+1,k,k}$ form a $(k-1)$ pair in $\mathcal{G}$, and $i+2k \leq L$.  (In particular, by Proposition \ref{case 1} we have $\gamma_j = 0$ if $j \geq L$.)  Consider the $3 \times 5$ submatrix of $N$ whose rows correspond to these two curves and $F_{L,k,k}$ and whose columns correspond to the curves $F_{p,1,1}$ for $p \in \{i,i+1,i+k+1,i+k+2,L \}$.  Then row reducing through this $(k-1)$ pair looks like the following:
\begin{displaymath}
\left( \begin{array}{ccccc}
1 & 2 & k-2 & k-3 & d\\
0 & 1 & k-1 & k-2 & d \\
a_{i} & b_{i} & 0 & 0 & c_{i} \\
\end{array} \right) \rightarrow 
\left( \begin{array}{ccccc}
1 & 2 & k-2 & k-3 & d\\
0 & 1 & k-1 & k-2 & d \\
0 & 0 & a_{i+1} & b_{i+1} & c_{i+1}
\end{array} \right),
\end{displaymath}
where
\begin{eqnarray}
a_{i+1} & = & (k-1)(a_i -b_i) + a_i \nonumber\\
b_{i+1} & = & (k-1)(a_i -b_i) + b_i \nonumber\\
c_{i+1} & = & c_i + (a_i-b_i)d 
\end{eqnarray}  
\item Suppose two curves $F_{i,k,k}$ and $F_{i+1,k,k}$ form a $(k-2)$ pair in $\mathcal{G}$, and $i+2k \leq L$.  (In particular, by Proposition \ref{case 1} we have $\gamma_j = 0$ if $j \geq L$.)  Consider the $3 \times 5$ submatrix of $N$ whose rows correspond to these two curves and $F_{L,k,k}$ and whose columns correspond to the curves $F_{p,1,1}$ for $p \in \{i,i+1,i+k,i+k+1,L \}$.  Then row reducing through this $(k-2)$ pair looks like the following:
\begin{displaymath}
\left( \begin{array}{ccccc}
1 & 2 & k-1 & k-2 & d\\
0 & 1 & k & k-1 & d \\
a_{i} & b_{i} & 0 & 0 & c_{i} 
\end{array} \right) \rightarrow 
\left( \begin{array}{ccccc}
1 & 2 & k-1 & k-2 & d \nonumber \\
0 & 1 & k & k-1 & d \nonumber\\
0 & 0 & a_{i+1} & b_{i+1} & c_{i+1}  
\end{array} \right), 
\end{displaymath}
where
\begin{eqnarray}
a_{i+1} & = & k(a_i -b_i) + a_i \nonumber \\
b_{i+1} & = & k(a_i -b_i) + b_i  \nonumber\\
c_{i+1} & = & c_i + (a_i-b_i)d
\end{eqnarray}  
\item Using the recurrences for $a_i$ and $b_i$ above, we obtain
  $a_{i+1}-b_{i+1} = a_i -b_i = -1$ for all $i$.
\item $a_i$, $b_i$, $(a_i-b_i)$, $a_{i+1}$, $b_{i+1}$, and $(a_{i+1}-b_{i+1})$ all have the same sign.
\end{enumerate}
\end{proposition}

Proposition \ref{k geq 3 recurrences} applies when $i\leq L-2k$ or $i\leq L-4k-2$.  We will row reduce $N$ using Proposition \ref{k geq 3 recurrences} as long as we can, and then compute the determinant of the remaining block directly.

Let $F_{L,k,k}$ be the last $k,k$-curve in the family $\Ghat$.  We need to consider four possibilities for how $F_{L,k,k}$ may fit into the family.  It could be the first curve in a pair (in which case $F_{L+1,k,k}$ is the curve that was dropped in the formation of $\Ghat$) or it could be the second curve in a pair.  This pair could be preceded by a $(k-1)$ pair or a $(k-2)$ pair.

Suppose that $j$ is odd.  Then by Proposition \ref{k geq 3 endings}, we see that $F_{L,k,k}$ is always the first member of a pair.  Then there are two possible final steps in the row reduction.  On the other hand, if $j$ is even, then by Proposition \ref{k geq 3 endings}, we see that $F_{L,k,k}$ is the second member of a pair.  Then there are two  possible final steps in the row reduction.  This leads to the following proposition.
\begin{proposition}  \label{k geq 3 matrices}
We write
\begin{equation}
d  = -c_{0} = \left\{ \begin{array}{ll}
-1 & \mbox{if $j$ is even} \\
-k+2 & \mbox{if $j$ is odd and $n,k$ have the same parity} \\
-k+1 & \mbox{if $j$ is odd and $n,k$ have opposite parity}
\end{array} \right.
\end{equation}
If $j$ is odd:
\begin{enumerate}
\item $F_{L,k,k}$ is preceded by a $(k-1)$ pair. Then row reducing $N$ yields a lower right $3 \times 3$ block of the form 
\begin{center}
\begin{tabular}{cc}
$\left( \begin{array}{ccc}
1 & 2 & 0 \\
0 & 1 & 1 \\
a_{T} & b_{T} & c_{T} \\
\end{array} \right)$ & 
$\left( \begin{array}{ccc}
1 & 2 & -1 \\
0 & 1 & 0 \\
a_{T} & b_{T} & c_{T} \\
\end{array} \right)$ \\
$n$,$k$ same parity & $n$,$k$ opposite parity\\
Example: \myneturltilde{http://www.math.uga.edu/~davids/agss/n74j11.m2}{$n=74$, $j=11$} & Example: \myneturltilde{http://www.math.uga.edu/~davids/agss/n73j11.m2}{$n=73$, $j=11$}
\end{tabular}
\end{center}
The determinant of this block is $t_{T} = ea_{T} - (e-1)b_{T}+c_{T}$ where $e=2$ or $1$, respectively.
\item $F_{L,k,k}$ is preceded by a $(k-2)$ pair. Then row reducing $N$ yields a lower right $3 \times 3$ block of the form 
\begin{center}
\begin{tabular}{cc}
$\left( \begin{array}{ccc}
1 & 2 & 1 \\
0 & 1 & 2 \\
a_{\tau} & b_{\tau} & c_{\tau} \\
\end{array} \right)$ & 
$\left( \begin{array}{ccc}
1 & 2 & 0 \\
0 & 1 & 1 \\
a_{\tau} & b_{\tau} & c_{\tau} \\
\end{array} \right)$ \\
$n$,$k$ same parity & $n$,$k$ opposite parity\\
Example: \myneturltilde{http://www.math.uga.edu/~davids/agss/n70j11.m2}{$n=70$, $j=11$} & Example: \myneturltilde{http://www.math.uga.edu/~davids/agss/n67j11.m2}{$n=67$, $j=11$}
\end{tabular}
\end{center}
The determinant of this block is $t_{\tau} = ea_{\tau} - (e-1)b_{\tau}+c_{\tau}$ where $e=3$ or $2$, respectively.
\end{enumerate}
If $j$ is even:
\begin{enumerate}
\item[(3)] $F_{L,k,k}$ and $F_{L-1,k,k}$ are preceded by a $(k-1)$ pair. Example: \myneturltilde{http://www.math.uga.edu/~davids/agss/n68j10.m2}{$n=68$, $j=10$}.  Then row reducing $N$ yields a lower right $4 \times 4$ block of the form 
\begin{displaymath}
\left( \begin{array}{cccc}
1 & 2 & k-2 & k-4  \\
0 & 1 & k-1 & k-3 \\
0 & 0 & 1 & 1 \\
a_{\tau} & b_{\tau} & 0  & c_{\tau} \\
\end{array} \right)
\end{displaymath}
The determinant of this block is $w_{\tau} = -2a_{\tau} +2b_{\tau}+c_{\tau}$.  
\item[(4)] $F_{L,k,k}$ and $F_{L-1,k,k}$ are preceded by a $(k-2)$ pair. Example: \myneturltilde{http://www.math.uga.edu/~davids/agss/n62j10.m2}{$n=62$, $j=10$}.  Then row reducing $N$ yields a lower right $4 \times 4$ block of the form 
\begin{displaymath}
\left( \begin{array}{cccc}
1 & 2 & k-1 & k-3  \\
0 & 1 & k & k-2 \\
0 & 0 & 1 & 1 \\
a_{\tau} & b_{\tau} & 0  & c_{\tau} \\
\end{array} \right)
\end{displaymath}
The determinant of this block is also $w_{\tau} = -2a_{\tau} +2b_{\tau}+c_{\tau}$.
\end{enumerate}
\end{proposition}
\begin{proof}
We use Proposition \ref{case 2} or Proposition \ref{gamma j} to compute the right hand column in each case.  
\end{proof}

We have already defined sequences $(a_i)$, $(b_i)$, and $(c_i)$ in Proposition \ref{k geq 3 endings}.  We use the defining relations for $t_{\tau}$ and $w_{\tau}$ given in Proposition \ref{k geq 3 matrices} above to define sequences $(t_i)$ and $(w_i)$  for all $i$. 

\begin{proposition} \label{tw recurrences} The sequences $(t_i)$ and $(w_i)$  satisfy the following recurrences:
\begin{equation}
t_{i+1}= \left\{ \begin{array}{ll}
t_{i} + (k+d-1)(a_i-b_i) & \mbox{if $s_i = k-1$} \\
t_{i} + (k+d)(a_i-b_i) & \mbox{if $s_i = k-2$} \\
\end{array} \right.
\end{equation}
\begin{equation}
w_{i+1}= \left\{ \begin{array}{ll}
w_{i} -(a_i-b_i) & \mbox{if $s_i = k-1$} \\
w_{i} -(a_i-b_i) & \mbox{if $s_i = k-2$} \\
\end{array} \right.
\end{equation}
\end{proposition}

\begin{proposition} 
\begin{enumerate}
\item $t_0$ is negative, and the sequence $(t_i)$ is nonincreasing, for any of the four pairs $(d,e) = (-k+1,1)$, $(-k+1,2)$, $(-k+2,2)$, or $(-k+2,3)$.
\item $w_0$ is positive, and the sequence $(w_{i})$  is nondecreasing.
\end{enumerate}
\end{proposition}
\begin{proof}
From direct computation and Proposition \ref{k geq 3 recurrences} we
know that $a_0 = -k$, $b_0=-k+1$, and $(a_i-b_i)= -1$.  Also $c_0=1$
if $j$ is even, $k-2$ if $j$ is odd and $n,k$ have the same parity, or
$k-1$ otherwise, and $d=-c_0$.  It is then easy to see that $t_0$ is
negative.  Then, by the recurrences established in Proposition \ref{tw
  recurrences}, we see that  $t_i$  weakly decreases.  On the other
hand $w_0$ is positive.   Then, by the recurrences established in
Proposition \ref{tw recurrences}, we see that  $w_i$  increases. 
\end{proof}

\begin{corollary} \label{k geq 3 independence theorem}
Suppose $k\geq 3$, $r >0$.   Then $\det(N) \neq 0$, and hence the family $\Ghat$ consists of independent curves.
\end{corollary}

\begin{theorem} \label{k geq 3 extremality}
Write $n=jk+r$ with $0 \leq r < j$.   Suppose $k\geq 3$.  Then the divisor $D^{n}_{1,j}$ spans an extremal ray of $\SymNef(\M_{0,n})$. 
\end{theorem}
\begin{proof}
By Propositions \ref{k geq 3 kills curves} and \ref{k geq 3 distinct}  and Corollary \ref{k geq 3 independence theorem}, we know that the family $\Ghat$ consists of $g-1$ independent curves, all of which intersect $D^{n}_{1,j}$ in degree zero.
\end{proof}

\begin{remark}
It is possible to solve the recurrences  above.  This gives the following
formulas for the $\det N$ when $r \neq 0$:
\begin{itemize}
\item If $j$ is even, $| \det N| = j/2$.
\item If $j$ is odd and $r$ is even, $| \det N| = j -r/2$.
\item If $j$ is odd and $r$ is odd, $| \det N | = (j-r)/2$.
\end{itemize}

These formulas have been checked experimentally for $n=6$ to $n=150$, and we also
proved it in a few cases.  They hold for $k=2$ as well as $k \geq 3$.  However, it did not seem worth the extra pages
to compute $\det N$ when showing $\det N \neq 0$ suffices to establish the
extremality of the divisors $D^{n}_{1,j}$.  Also, we wanted to keep the
exposition for the cases $k\geq 3$ and $k=2$ as similar as possible, as we
only outline the proof for the case $k=2$ in the next section.
\end{remark}

\section{\texorpdfstring{The family $\mathcal{G}$ when $k=2$}{The
    family G when k=2}}\label{k=2 structure}
As in the previous two sections, write $n=jk+r$ with $0 \leq r < j$.  Now we consider the case $k=2$.  There are many parallels between the definition and structure of the family when $k=2$ and $k \geq 3$, but ultimately we felt it was easier to write up these two cases separately.

In contrast with the previous case, we will simply outline the proof when $k=2$, omitting many details.

\begin{definition}[Definition of $\mathcal{G}$ when
  $k=2$.]\label{ThirdFamily} \label{k=2 family}
Write $n=jk+r$ with $0 < r < j$ and $k=2$.  Suppose $j \geq 6$.  (For the cases $r=0$ and $j\leq 5$ see 
Section \ref{special cases section}.)

We define families of
curves $\mathcal{G}$ and $\Ghat$ as follows:
If $\frac{(i+2)j}{n} \in \Z$, then set $\mathcal{G}(i) = F_{i,3,3}$.
If $\frac{(i+2)j}{n} \not\in \Z$ and $D^{n}_{1,j} \cdot
F_{i,1,1}=0$, then set $\mathcal{G}(i) = F_{i,3,3}$.  If $\frac{(i+2)j}{n} \not\in \Z$ and $D^{n}_{1,j} \cdot
F_{i,1,1} \neq 0$ and $D^{n}_{1,j} \cdot
F_{i,2,2} = 0$, then set $\mathcal{G}(i) = F_{i,2,2}$.  Otherwise, if $\frac{(i+2)j}{n} \not\in \Z$ and $D^{n}_{1,j} \cdot
F_{i,1,1} \neq 0$ and $D^{n}_{1,j} \cdot
F_{i,2,2} \neq 0$, then set $\mathcal{G}(i) = F_{i,3,3}$. 

We set $\mathcal{G} = \{ \mathcal{G}(i) : 1 \leq i \leq g\}$.  Let $p$
be the largest index such that $\mathcal{G}(i) \neq F_{i,1,1}$.  Then
we set $\Ghat= \{ \mathcal{G}(i) : 1 \leq i \leq g, i \neq p\}$.

\end{definition}

\begin{example} Let $n=28$ and $j=11$.  We abbreviate and write $D_{11}$ for $D^{28}_{1,11}$.  Then $g=13$, so to prove that $D_{11}$ is extremal, we need $12$ independent $\F$-curves which intersect $D_{11}$ in degree zero.  Also $k = \lfloor 28/11 \rfloor = 2$, so if necessary we will replace curves of the form $F_{1,1,m}$ with curves of the form $F_{2,2,m}$ or $F_{3,3,m}$.  

We compute intersection numbers and choose curves below:
\begin{displaymath}
\begin{array}{lll}
\underline{F_{m,1,1}  \cdot D_{11}}& \underline{\mbox{Replace by }}& \underline{\mbox{Replace by }} \\
F_{1,1,1}  \cdot D_{11} = 5 & F_{1,2,2} \cdot D_{11} = 1 & \fcolorbox{red}{white}{$F_{1,3,3} \cdot D_{11} = 0$} \\ 
F_{2,1,1} \cdot D_{11} = 6 & \fcolorbox{red}{white}{$F_{2,2,2} \cdot D_{11} = 0$}   & \\ 
\fcolorbox{red}{white}{$F_{3,1,1}  \cdot D_{11} = 0 $}&  & \\ 
F_{4,1,1} \cdot D_{11} =10 & \fcolorbox{red}{white}{$F_{4,2,2} \cdot D_{11} = 0 $}&\\ 
F_{5,1,1} \cdot D_{11} = 1 & \fcolorbox{red}{white}{$F_{5,2,2} \cdot D_{11} = 0 $}&  \\ 
F_{6,1,1} \cdot D_{11} = 4 & F_{6,2,2} \cdot D_{11} = 2 & \fcolorbox{red}{white}{$F_{6,3,3} \cdot D_{11} = 0 $ }\\ 
F_{7,1,1} \cdot D_{11} = 7 & \fcolorbox{red}{white}{$F_{7,2,2} \cdot D_{11} = 0 $}& \\ 
\fcolorbox{red}{white}{$F_{8,1,1} \cdot D_{11} = 0 $}&  &\\ 
F_{9,1,1} \cdot D_{11} = 9 &  \fcolorbox{red}{white}{$F_{9,2,2} \cdot D_{11} = 0  $}& \\ 
F_{10,1,1} \cdot D_{11} = 2 &  \fcolorbox{red}{white}{$F_{10,2,2} \cdot D_{11} = 0  $}& \\ 
F_{11,1} \cdot D_{11} = 3 & F_{11,2,2} \cdot D_{11} = 3  & \fcolorbox{red}{white}{$F_{11,3,3} \cdot D_{11} = 0  $}\\ 
F_{12,1,1} \cdot D_{11} = 8 & \fcolorbox{red}{white}{$F_{12,2,2} \cdot D_{11} = 0 $}& \\ 
\fcolorbox{red}{white}{$F_{13,1,1} \cdot D_{11} = 0 $}& & 
\end{array}
\end{displaymath}
Then, we drop the last curve that is not of the form $F_{i,1,1}$, which is $F_{12,2,2}$.  Thus, for $n=28$, 
\begin{displaymath}\Ghat= \{ F_{1,3,3}, F_{2,2,2}, F_{3,1,1}, F_{4,2,2}, F_{5,2,2}, F_{6,3,3}, F_{7,2,2}, F_{8,1,1}, F_{9,2,2}, F_{10,2,2}, F_{11,3,3}, F_{13,1,1} \},
\end{displaymath}
and $D_{11}$ intersects every member of this family in degree zero.  

\end{example}

\begin{proposition} Every member of $\mathcal{G}$ intersects $D^{n}_{1,j}$ in degree zero.
\end{proposition}

The family $\familyG$ may contain $1,1$-curves, $2,2$-curves,
and $3,3$-curves.  In the next proposition, we aim to describe the
patterns in which these appear.

\begin{proposition}
The curves in $\familyG$ may be grouped as follows:
\begin{enumerate}
\item $\familyG(i)=F_{i,2,2}$ and  $\familyG(i+1)=F_{i+1,2,2}$.  We call this pattern a $22$.
\item $\familyG(i)=F_{i,3,3}$,  $\familyG(i+1)=F_{i+1,2,2}$, and $\familyG(i+2)=F_{i+2,1,1}$.  We call this pattern a $321$.
\end{enumerate}
\end{proposition}

\begin{definition} We define a \emph{sequence of pairs and triples}
  $(s_i)$ when $k=2$.  (Compare this to Definition \ref{k geq 3 s} which
  applies when  $k \geq 3$.)  If the $i^{th}$ group is of the first
  type above, we set $s_i=22$.  If the $i^{th}$ group is of the second
  type above, we set $s_i=321$.  
\end{definition}

Given $n$ and $j$ we can determine how the family $\familyG$ begins.
\begin{proposition}\label{k=2 s_0}
Write $n=jk+r$ with $0 \leq r < j$ and $k=2$.  Then if $j < 2r$, we have $s_0=321$, and if $j \geq 2r$, we have $s_0=22$.  
\end{proposition}

Given $n$ and $j$, we can determine how the family $\familyG$ ends.  
\begin{proposition}\label{k=2 endings}
Write $n=jk+r$ with $0 \leq r < j$ and $k=2$.   Then there are thirteen possible endings for the family $\familyG$.  These are listed in Table \ref{data table} below.  
\end{proposition}

\begin{proposition}
The family $\Ghat$ consists of $g-1$ distinct curves.
\end{proposition}

\section{\texorpdfstring{Extremality of the divisors $D^{n}_{1,j}$ when $k =2$}{Extremality of the divisors Dn1j when k=2}}\label{k=2 independence}

\begin{definition}
We let $L_2$ be the index of the last $2,2$-curve in
$\Ghat$ and let $L_3$ be the index of the last $3,3$-curve in
$\Ghat$.  Let $L=\min \{ L_2, L_3\}$.
\end{definition}

It will be convenient to define one more sequence, $(\beta_i)$:
\begin{definition}  \label{betai def}
\begin{displaymath} \beta_{i} = \# \{ p : 0 \leq p<i, s_p = 321    \}
 \end{displaymath} 
\end{definition}

The proof of extremality will now break into several cases, depending
on the pattern of the curves in the family $\familyG$.  The strategy is similar in every case:
\begin{enumerate}
\item[\textbf{Step 1.}] Find a recurrence
relation satisfied by the rows $L_2$ and $L_3$ as one row reduces these rows 
through a $321$ or a $22$.
\item[\textbf{Step 2.}] Define a sequence $x_i$ whose final
member $x_{\tau}$ computes $\det N$.
\item[\textbf{Step 3.}] Show that the sequence $(x_i)$ 
is absolutely nondecreasing, and $x_1 \neq 0$.
Thus $\det N \neq 0$.
\end{enumerate}



In the next definition and proposition, we carry out Step 1.  We define
sequences $(a_i)$, $(b_i)$, $(d_i)$, $(e_i)$, $(A_i)$, $(B_i)$,
$(D_i)$, and $(E_i)$, and find recurrence relations for these
sequences as one row reduces rows $L_2$ and $L_3$ through a $321$ or a $22$.

\begin{definition}[Definition of the sequences $a-d$ and $A-E$]
\begin{enumerate}
\item If we row reduce row $L_2$ through two columns corresponding to a 22, or two columns corresponding to a 321, there are at most two leading nonzero entries.  We denote these $a_i$ and $b_i$ respectively.
\item Row $L_3$ begins with three nonzero leading entries, which we denote $A_0$, $B_0$, $C_0$.  After row reducing through the first two columns, there are at most two nonzero leading entries, which we denote $A_1$ and $B_1$.  If we row reduce rows $L_2$ and $L_3$ through two columns corresponding to a 22, or two columns corresponding to a 321, we again have at most two leading nonzero entries, which we denote $A_i$ and $B_i$.  
\item  Suppose $N$ has $m$ columns.  We define sequences $(d_i)$ and
  $(e_i)$ as the values in row $L_2$ and the columns specified in Table \ref{data table}  as
  the matrix $N$ is row reduced.  Similarly, we define sequences
  $(D_i)$ and $(E_i)$ as the entries in row $L_3$ and the column
  specified in Table \ref{data table} as $N$ is row reduced. 
\end{enumerate}
\end{definition}

\begin{proposition} \label{k=2 recurrences}  
\begin{enumerate}
\item \emph{Recursion when row reducing through a 22.} Suppose $\mathcal{G}(i)=F_{i,2,2}$ and
  $\mathcal{G}(i+1)=F_{i+1,2,2}$,  and $i+4 \leq L_2$.  Consider
  the $4 \times 6$ submatrix of $N$ whose rows correspond to these two
  curves, $F_{L_2,2,2}$, and $F_{L_3,3,3}$,  and whose columns
  correspond to the curves $F_{p,1,1}$ for $p \in \{i,i+1,i+2,i+3 \}$
  plus the columns containing $d_i$ and $e_i$.  Then row reducing through this pair looks like the following:
\begin{displaymath}
\left( \begin{array}{cccccc}
1 & 2 &1 &0  & -d_0 & -e_0\\
0 & 1 & 2 & 1 &  -d_0 & -e_0 \\
a_{i} & b_{i} & 0 & 0 & d_{i} & e_i \\
A_{i} & B_{i} & 0 & 0 & D_{i} & E_i 
\end{array} \right) \rightarrow 
\left( \begin{array}{cccccc}
1 & 2 &1 &0 &  -d_0 & -e_0\\
0 & 1 & 2 & 1  & -d_0 & -e_0 \\
0 & 0 & a_{i+1} & b_{i+1} & d_{i+1} & e_{i+1} \\
0 & 0 & A_{i+1} & B_{i+1} & D_{i+1} & E_{i+1} \\
\end{array} \right)
\end{displaymath}%
We have
\begin{eqnarray*}
a_{i+1} & = & 3a_i - 2b_i\\
b_{i+1} & = & 2a_i - b_i\\
d_{i+1} & =&   d_i- d_0(a_i - b_i) \\
e_{i+1} & =& e_{i} - e_0(a_i -  b_i)
\end{eqnarray*}
for $i\geq 0$, and similar recurrences hold for $A_i, B_i, D_i, E_i$ for $i\geq 1$.  
\item \emph{Recursion when row reducing through a 321.}  Suppose $\mathcal{G}(i)=F_{i,3,3}$,
 $\mathcal{G}(i+1)=F_{i+1,2,2}$, and $\mathcal{G}(i+2)=F_{i+2,1,1}$, and $i+4 \leq L_2$.  Consider
  the $4 \times 6$ submatrix of $N$ whose rows correspond to
  $F_{i,3,3}$, $F_{i+1,2,2}$, $F_{L_2,2,2}$, and $F_{L_3,3,3}$,  and
  whose columns correspond to the curves $F_{p,1,1}$ for $p \in
  \{i,i+1,i+3,i+4 \}$ together with the columns containing $d_i$ and $e_i$.  Then row reducing through this pair looks like the following:
\begin{displaymath}
\left( \begin{array}{cccccc}
1 & 2 &2 &1 & -D_0 & -E_0\\
0 & 1 & 1 & 0 &  -d_0 & -e_0 \\
a_{i} & b_{i} &0&  0 & d_{i} & e_i \\
A_{i} & B_{i} & 0 &  0 & D_{i} & E_i 
\end{array} \right) \rightarrow 
\left( \begin{array}{cccccc}
1 & 2 &2 &1 &  -D_0 & -E_0\\
0 & 1 & 1 & 0  & -d_0 & -e_0 \\
0 & 0 & a_{i+1} & b_{i+1} &  d_{i+1} & e_{i+1} \\
0 & 0 & A_{i+1} & B_{i+1} &  D_{i+1} & E_{i+1} \\
\end{array} \right)
\end{displaymath}
We have
\begin{eqnarray*}
a_{i+1} & = & -b_i\\
b_{i+1} & = & -a_i\\
d_{i+1} & =&   d_i  + (D_0-2d_0) a_i + d_0 b_i \\
e_{i+1} & =& e_{i} +(E_0-2e_0)a_i + e_0 b_i
\end{eqnarray*}
for $i \geq 0$, and similar recurrences hold for $A_i, B_i, D_i, E_i$ for $i \geq 1$.  
\item In particular, the recurrences for $a_i$, $b_i$, $A_i$, and $B_i
  $ are case independent. 
\item $a_{i+1}-b_{i+1} = a_{i}-b_{i} = -1$.
\item Suppose $i \geq 1$.  Then $A_i = B_i = (-1)^{\beta_i+1} 2$.
\end{enumerate}
\end{proposition}

\subsection{Proof of extremality when \texorpdfstring{$k=2$}{k=2}: an example}
We will present the proof of extremality in Case (1) as an example before we give the
general proof.

Suppose that $n$ is even, $j$ is odd, and $j \leq \frac{5}{4}r$. Then by Proposition \ref{k=2 endings} and Table \ref{data table}, the sequence
  $(s_i)$ ends $321,321,3\xout{2}1$.  Thus $L_3 = g-2$ and $L_2 =  g-3$.  The sequences $(d_i)$ and $(D_i)$ are defined by the second
  to last column (i.e. the column corresponding to the curve
  $F_{g-3,2,2}$) and the sequences $(e_i)$ and $(E_i)$ are defined by the last column (i.e. the column corresponding to the curve
  $F_{g-2,3,3}$).

For example, here is the matrix $N$ for $n=48, j=17$.  The sequence $(s_i)$ is $321$, $321$, $321$, $321$, $22$, $321$, $321$, $3\xout{2}1$.  

\begin{displaymath}
\left(
\begin{array}{ccccccccccccccc}
1 & 2 & 2 & 1 & 0 & 0 & 0 & 0 & 0 & 0 & 0 & 0 & 0 & 0 & -2\\
0 & 1 & 1 & 0 & 0 & 0 & 0 & 0 & 0 & 0 & 0 & 0 & 0 & -1 & -1\\
0 & 0 & 1 & 2 & 2 & 1 & 0 & 0 & 0 & 0 & 0 & 0 & 0 & 0 & -2\\
0 & 0 & 0 & 1 & 1 & 0 & 0 & 0 & 0 & 0 & 0 & 0 & 0 & -1 & -1\\
0 & 0 & 0 & 0 & 1 & 2 & 2 & 1 & 0 & 0 & 0 & 0 & 0 & 0 & -2\\
0 & 0 & 0 & 0 & 0 & 1 & 1 & 0 & 0 & 0 & 0 & 0 & 0 & -1 & -1\\
0 & 0 & 0 & 0 & 0 & 0 & 1 & 2 & 2 & 1 & 0 & 0 & 0 & 0 & -2\\
0 & 0 & 0 & 0 & 0 & 0 & 0 & 1 & 1 & 0 & 0 & 0 & 0 & -1 & -1\\
0 & 0 & 0 & 0 & 0 & 0 & 0 & 0 & 1 & 2 & 1 & 0 & 0 & -1 & -1\\
0 & 0 & 0 & 0 & 0 & 0 & 0 & 0 & 0 & 1 & 2 & 1 & 0 & -1 & -1\\
0 & 0 & 0 & 0 & 0 & 0 & 0 & 0 & 0 & 0 & 1 & 2 & 2 & 1 & -2\\
0 & 0 & 0 & 0 & 0 & 0 & 0 & 0 & 0 & 0 & 0 & 1 & 1 & -1 & -1\\
0 & 0 & 0 & 0 & 0 & 0 & 0 & 0 & 0 & 0 & 0 & 0 & 1 & 2 & 0\\
-2 & -1 & 0 & 0 & 0 & 0 & 0 & 0 & 0 & 0 & 0 & 0 & 0 & 1 & 1\\
-2 & -3 & -1 & 0 & 0 & 0 & 0 & 0 & 0 & 0 & 0 & 0 & 0 & 0 & 2
\end{array}
\right)
\end{displaymath}

We have $d_0 = 1$, $e_0 = 1$, $D_0 = 0$, $E_0 = 2$.

As we row reduce rows $L_2$ and $L_3$ of $N$, the recurrences of
Proposition \ref{k=2 recurrences} hold until we reach the lower right  $5 \times 5$ block.  (We can see this in the example above for $n=48,j=17$ as follows.  The first ten rows all have zeroes to the left of the diagonal, followed by at most four nonzero entries starting on the diagonal, followed by zeroes until the last two columns.  Thus, the recurrences apply as long as the nonzero entries do not reach the last two columns, i.e. when we reach the lower right $5 \times 5$ block.)

After all these row reductions, the lower right $5 \times 5$ block is 
\begin{displaymath}
\left( \begin{array}{ccccc}
1 & 2 & 2 & 1 & -2\\
0 & 1 & 1 & -1 & -1\\
0 & 0& 1& 2 & 0\\
a_{\tau} & b_{\tau} & 0& d_{\tau} & e_{\tau} \\
A_{\tau} & B_{\tau} & 0 & D_{\tau} & E_{\tau} \\
\end{array} \right).
\end{displaymath}
The determinant of this block is 
\begin{displaymath}
x_{\tau} = (3b_{\tau} +3e_{\tau} )A_{\tau}  + (-3a_{\tau} +d_{\tau} -3e_{\tau} )B_{\tau}  + (-b_{\tau} -e_{\tau} )D_{\tau}  + (-3a_{\tau} +3b_{\tau} +d_{\tau} )E_{\tau} 
\end{displaymath}
We use this relation to define a sequence $(x_i)$ for all $i\geq 1$ in terms
of the sequences $(a_i)$,  $(b_i)$, $(d_i)$, $(e_i)$, $(A_i)$,
$(B_i)$, $(D_i)$, $(E_i)$, which were previously defined.

\begin{proposition} \label{22321 ABCD prop} 
Suppose $n=jk+r$ with $0 \leq r < j$ and $k=2$.  Suppose $n$ is even, $j$ is odd, $j \leq \frac{5}{4}r$, so that we are in Case (1).  
\begin{enumerate}
\item Suppose $i \geq 1$. Then $D_{i} = (-1)^{\beta_i+1}$.
\item Suppose $i \geq 1$.  Then $E_i=(-1)^{\beta_i}$.
\item Using the previous formulas gives a simplified formula for $x_i$:
\[  x_i = (-1)^{\beta_i+1} (-b_i+ d_i-e_i+3)
\]
\end{enumerate}
\end{proposition}
\begin{proof}
Since $j < 2r$, we have that $s_0=321$.  Then we may compute directly by row reducing that $D_1 = 1$ and $E_1=-1$. Then, for $i \geq 1$, we may apply the recurrence
relations for $D_i$ and the formulas for $A_i$ and $B_i$ obtained in Proposition \ref{k=2 recurrences} above.

The third claim follows once we have the previous two formulas for
$D_i$ and $E_i$.
\end{proof}

\begin{proposition}
\begin{enumerate}
\item $x_1 =5$.
\item The sign of $x_i$ is  $(-1)^{\beta_i+1}$, i.e. $-b_i+ d_i-e_i+3
  \geq 0$.
\item For $i \geq 1$,
\begin{displaymath} 
| x_{i+1} | = \left\{
\begin{array}{ll}
|x_{i}| + 2 & \mbox{ if $s_i=22$}\\
|x_{i}| + 1 & \mbox{ if $s_i=321$}
\end{array} \right.
\end{displaymath}
\end{enumerate}
\end{proposition}
\begin{proof}

We prove the first statement by a direct calculation.  Since $j < 2r$, we know that $s_0 =321$.  Thus we can compute $(a_1, b_1, d_1, e_1) = (1,2,4,0)$ and $(A_1, B_1, D_1, E_1) = (2,2,1,-1)$, and hence $x_1= (-1)^{2}(-2+4-0+3)=5$.

We prove the remaining statements by induction on the number of
$321$ and $22$'s encountered, with $i=1$ as the base case.  The calculation above shows that the claimed sign for $x_1$ is correct.

For the induction step,  if $s_i=22$, so that we row reduce through a 22, we can use the recurrences of Proposition \ref{k=2 recurrences} as follows:
\begin{eqnarray*}
x_{i+1} & = & (-1)^{\beta_{i+1}+1}(-b_{i+1}+d_{i+1}-e_{i+1}+3) \\%
& = & (-1)^{\beta_{i}+1}(-(b_i-2) + (d_i+d_0) - (e_i+e_0)+3)\\
& = & (-1)^{\beta_{i}+1}(-b_i+d_i-e_i+3+2)\\
& = & x_i+ (-1)^{\beta_{i}+1}2
\end{eqnarray*}
By the induction hypothesis, $x_i$ and $(-1)^{\beta_{i}+1}2$ have the same sign, so we have $|x_{i+1}| = |x_i|+2$.

Similarly, if $s_i=321$, so that we row reduce through a 321, then again we apply  the recurrences of Proposition \ref{k=2 recurrences}:
\begin{eqnarray*}
x_{i+1} & = & (-1)^{\beta_{i+1}+1}(-b_{i+1}+d_{i+1}-e_{i+1}+3) \\%
& = & (-1)^{\beta_{i}}(-(-a_ib_i-2) + (d_i-2a_i+b_i) - (e_i+b_i)+3)\\
& = & (-1)^{\beta_{i}}(-a_i+d_i-e_i+3)\\
& = & (-1)^{\beta_{i}}(-b_i+d_i-e_i+3+1)\\
& = & -(x_i+ (-1)^{\beta_{i}+1})
\end{eqnarray*}
By the induction hypothesis, $x_i$ and $(-1)^{\beta_{i}+1}$ have the same sign.  We have $|x_{i+1}| = |x_i|+1$.
\end{proof}

Indeed, in the example $n=48, j=17$ above, the sequence $x_i $ is $5, -6, 7, -8, -10$ for $i=1,\ldots,5$.

\subsection{Proof of extremality when \texorpdfstring{$k=2$}{k=2}: the general argument}

We refer to Table \ref{data table} for all the data that is necessary for each case of the general argument.

Given $n$ and $j$, we first compute how the family $\mathcal{G}$ ends.  This leads to thirteen different cases.  Once we know the ending, we can compute $L_2$ and $L_3$.  

Next, for each case, we can determine which columns define the
sequences $d,D$ and $e,E$.  Note by Proposition \ref{cij summary} that when we write $F_{L_2,2,2}$ and $F_{L_3,3,3}$ in the basis $\{ F_{q,1,1} \}$, these rows begin with two or three nonzero entries followed by several zeroes, followed by a nonzero ``tail'' of length less than 4.  These two tails are then propagated to all the previous rows as we form the matrix $N$ from the matrix $M$.  Let $L:=\min \{ L_2, L_3 \}$. It follows that the upper right $(L-1) \times (m-L+1)$ block of $N$ has at most two
  independent columns.  We select the leftmost two independent columns to define $d,D,e,E$.  The results are listed in Table \ref{data table}.

Next, for each case, we can determine when the recurrences of Proposition \ref{k=2 recurrences} will break down by finding the last row before row $L$ whose nonzero entries after the diagonal do not reach column $d,D$ and column $e,E$.  

Thus we can use the recurrences of Proposition \ref{k=2 recurrences} to row reduce $N$ until we arrive at one of the thirteen matrices given in Figure \ref{thirteen matrices}.  Since $N$ is upper triangular with 1's on the diagonal for the first $L-1$ rows, the determinant of $N$ is computed by the determinant of this remaining lower right block.   We define $x_{\tau}$ as the determinant of this block and compute $x_{\tau}$ in terms of $a_{\tau}$, $b_{\tau}$, $d_{\tau}$, $e_{\tau}$, $A_{\tau}$, $B_{\tau}$, $D_{\tau}$, and $E_{\tau}$.  We then use this relation to define $x_i$ for $i \geq 1$ in terms of $a_{i}$, $b_{i}$, $d_{i}$, $e_{i}$, $A_{i}$, $B_{i}$, $D_{i}$, and $E_{i}$.  We write all of these expressions in the format $x_i = (-1)^{\beta_{i}+1}f(a_i,b_i,d_i,e_i)$ and list $f$ in Table \ref{data table}.  

\begin{proposition} \label{k=2 x_i nondecreasing}
The sequence $(x_i)$ is absolutely nondecreasing, and $x_1 \neq 0$.  
\end{proposition}

\begin{corollary}
$\det N = x_{\tau} \neq 0$.  Therefore the family $\Ghat$ is independent.  
\end{corollary}

\begin{proof}[Proof of Proposition \ref{k=2 x_i nondecreasing}]  We refer to the third block of Table \ref{data table} throughout the discussion below.

The proof is by induction on $i$ (the number of $321$ blocks or $22$ blocks) with $i=1$ as the base case.  

To check the base case, in each case (1)-(13), we can compute $a_1$, $b_1$, $d_1$, $e_1$, $A_1$, $B_1$, $D_1$, and $E_1$ and hence $x_1$, and verify that $x_1 \neq 0$ and has the sign listed in Table \ref{data table}.

For the induction step, we can use the recurrences of Proposition \ref{k=2 recurrences} to establish a relationship between $x_{i+1}$ and $x_i$ when we row reduce through a 22 or 321 respectively.  By carefully analyzing the resulting relations together with the signs of $x_{i+1}$ and $x_i$, we see that $|x_{i+1}| \geq |x_{i}|$.  The result follows.
\end{proof}

\section*{Appendix: Data for the proof of extremality when $k=2$}

\begin{figure}[h] \label{data table}
\caption{Tables of data for the proof of extremality when $k=2$}
\begin{center}
\begin{tabular}{|lllllll|} \hline
Case                                                                     & $s_0$& Ending                       & $L_2$ & $L_3$ & $d$,$D$  & $e$,$E$  \\ 
                                                                                        &          &                                   &           &           &     col.           & col.  \\ \hline
 1 : $n$ even, $j$ odd, $j \leq 5r/4$                                &  321 &$321,321,3\xout{2}1$ & $m\!-\!1$ & $m$ & $m\!-\!1$ & $m$ \\ 
 2 : $n$ even, $j$ odd, $\frac{5}{4}r < j \leq \frac{3}{2}r$ &  321&$22,321,3\xout{2}1$ & $m\!-\!1$& $m$ & $m\!-\!1$ & $m$ \\ 
 3 : $n$ even, $j$ odd, $\frac{3}{2}r < j \leq \frac{5}{2}r$ &        & $321,22,3\xout{2}1$ & $m\!-\!1$ & $m$ & $m\!-\!1$ & $m$ \\
 4 : $n$ even, $j$ odd, $j > \frac{5}{2}r$                          &  22   &$22,22,3\xout{2}1$   & $m\!-\!1$ & $m$ & $m\!-\!1$ & $m$ \\ \hline
 5 : $n$ even, $j$ even, $j \leq 2r$                                   &  &$321,22,3\xout{2}$   & $m\!-\!1$ & $m$ & $m\!-\!2$ & $m\!-\!1$ \\ 
 6 : $n$ even, $j$ even, $j > 2r$                                       &  22  &$22,22,3\xout{2}$     & $m\!-\!1$ & $m$ & $m\!-\!2$ & $m\!-\!1$ \\ \hline
 7 : $n$ odd, $j$ even, $j \leq \frac{4}{3}r$                      &  321 &$321,321,(\xout{3}$ or $\xout{2})$& $m$ & $m\!-\!1$ & $m$ & $m\!-\!1$ \\
 8 : $n$ odd, $j$ even, $\frac{4}{3}r < j \leq 2r$              &   &$22,321,(\xout{2}$ or $\xout{3})$& $m$ & $m\!-\!1$ & $m$ & $m\!-\!1$ \\
 9 : $n$ odd, $j$ even, $2r < j \leq 4r$                             &  22   &$22,321,22,\xout{2}$ & $m$ & $m-3$ & $m$ & $m-3$ \\
10 : $n$ odd, $j$ even, $j >4r$                                  &  22  &$22,321,2^{\delta},\xout{2}$, & $m$ & $m\!-\!\delta\!-\!1$ & $m\!-\!\delta\!-\!1$ & $m$  \\
                                                                                        &        &$\delta = 2 \lceil \frac{j}{2r} \rceil -2$ & & & & \\ \hline
11 : $n$ odd, $j$ odd, $j \leq \frac{5r}{3}$                       & 321 &$321,321,2\xout{2}$  &  $m$ & $m\!-\!2$ & $m\!-\!2$ & $m$ \\
12 : $n$ odd, $j$ odd, $\frac{5r}{3} < j \leq 3r$               &        &$22,321,2\xout{2}$    & $m$  & $m\!-\!2$ & $m\!-\!2$ & $m$  \\
13 : $n$ odd, $j$ odd, $j > 3r$                                         & 22  &$22,321,2^{\delta}\xout{2}$& $m$ & $m\!-\!\delta\!-\!1$ & $m\!-\!\delta\!-\!1$ & $m$ \\ 
                                                                                         &       &$\delta=2\lceil \frac{j-r}{2r}\rceil-1$ & & & & \\ \hline
\end{tabular}
\end{center}
\end{figure}

\newpage
\begin{figure}[h] 
\caption{Tables of data for the proof of extremality when $k=2$, continued}
\begin{center}
\begin{tabular}{|lllllllll|} \hline
Case & $d_0$ & $e_0$ & $D_0$ & $E_0$ & $D_i$ & $E_i$ & $(-1)^{-\beta_i-1} x_i$ & Example $(n,j)$\\\hline 
 1 & 1 & 1 & 0 & 2 & $(-1)^{\beta_i+1}$ & $(-1)^{\beta_i}$      & $-b_i + d_i -e_i+3$ &  \myneturltilde{http://www.math.uga.edu/~davids/agss/n62j21.m2}{$(62, 21)$} \\
 2 & 1 & 1 & 0 & 2 & $(-1)^{\beta_i+1}$ & $(-1)^{\beta_i}$      & $a_i- d_i+e_i-3$      &  \myneturltilde{http://www.math.uga.edu/~davids/agss/n62j23.m2}{$(62, 23)$} \\
 3 & 1 & 2 & 0 & 2 &  $(-1)^{\beta_i+1}$ & 0                         & $a_i-2d_i+e_i-3$     & \myneturltilde{http://www.math.uga.edu/~davids/agss/n66j25.m2}{$(66, 25)$}\\
 4 & 1 & 2 & 0 & 2 & $(-1)^{\beta_i+1}$ & 0                          & $-a_i+2d_i-e_i+4$  &  \myneturltilde{http://www.math.uga.edu/~davids/agss/n62j27.m2}{$(62, 27)$} \\ \hline
 5 & 0 & 1 & 1 & 2 &  $(-1)^{\beta_i+1}$ & $(-1)^{\beta_i+1}$ & $2-d_i+e_i$             &  \myneturltilde{http://www.math.uga.edu/~davids/agss/n62j2.m2}{$(62, 22)$} \\
 6 & 0 & 1 & 1 & 2 & $(-1)^{\beta_i+1}$ & $(-1)^{\beta_i+1}$  & $-2+d_i-e_i$           &  \myneturltilde{http://www.math.uga.edu/~davids/agss/n62j26.m2}{$(62, 26)$}  \\ \hline
 7 & 0 & 1 & 1 & 2 & $(-1)^{\beta_i+1}$ & $(-1)^{\beta_i+1}$  & $-2+d_i-e_i$           &  \myneturltilde{http://www.math.uga.edu/~davids/agss/n63j22.m2}{$(63, 22)$} \\
 8 & 0 & 1 & 1 & 2 & $(-1)^{\beta_i+1}$ & $(-1)^{\beta_i+1}$  & $2-d_i+e_i$             & \myneturltilde{http://www.math.uga.edu/~davids/agss/n65j26.m2}{$(65, 26)$} \\
 9 & 0 & 1 & 1 & 1 & $(-1)^{\beta_i}$    & 0                          & $e_i+3$                   & \myneturltilde{http://www.math.uga.edu/~davids/agss/n63j26.m2}{$(63, 26)$} \\
10& 0 & 1 & 1 & 0 & $(-1)^{\beta_i}$ & $(-1)^{\beta_i+1}$      & $2+d_i+e_i+\delta/2$ &  \myneturltilde{http://www.math.uga.edu/~davids/agss/n61j28.m2}{$(61, 28)$}   \\ \hline
11& 0 & 1 & 1 & 2 & $(-1)^{\beta_i+1}$ & $(-1)^{\beta_i+1}$  & $1-b_i-d_i-e_i$            & \myneturltilde{http://www.math.uga.edu/~davids/agss/n63j23.m2}{$(63, 23)$} \\
12& 0 & 1 & 1 & 2 & $(-1)^{\beta_i+1}$ & $(-1)^{\beta_i+1}$   & $-1+a_i+d_i+e_i$        & \myneturltilde{http://www.math.uga.edu/~davids/agss/n61j25.m2}{$(61, 25)$} \\
13& 0 & 1 & 1 & 0 & $(-1)^{\beta_i}$ & $(-1)^{\beta_i+1}$      & $a_i+3d_i+e_i-(\delta+1)/2$  & \myneturltilde{http://www.math.uga.edu/~davids/agss/n61j27.m2}{$(61, 27)$}\\  \hline
\end{tabular}
\end{center}
\end{figure}
\mbox{} \vspace{12pt}

\begin{figure}[h] 
\caption{Tables of data for the proof of extremality when $k=2$, continued}
\begin{center}
\begin{tabular}{|llll|} \hline
Case & $\operatorname{sign}(x_i)$ &   $x_{i+1}$ & $x_{i+1}$ \\ \hline
1 & $(-1)^{\beta_i+1}$ & $x_i+ (-1)^{\beta_i+1}2$ & $-( x+(-1)^{\beta_i+1})$ \\
2 & $(-1)^{\beta_i}$ &  $x_i+ (-1)^{\beta_i}2$  & $-(x_i+(-1)^{\beta_i})$\\
3 & $(-1)^{\beta_i}$ &  $x_i + (-1)^{\beta_i}2$ & $-(x_i+(-1)^{\beta_i})$\\
4 & $(-1)^{\beta_i+1}$ & $x_i+ (-1)^{\beta_i+1}2$ &   $-(x_i+(-1)^{\beta_i+1})$ \\ \hline
5 & $(-1)^{\beta_i+1}$ & $x_i+ (-1)^{\beta_i+1}$ & $-(x_i+(-1)^{\beta_i+1})$ \\
6 & $(-1)^{\beta_i}$  & $x_i+ (-1)^{\beta_i}$  & $-(x_i+(-1)^{\beta_i})$\\ \hline
7 & $(-1)^{\beta_i}$  & $x_i+ (-1)^{\beta_i}$  & $-(x_i+(-1)^{\beta_i})$\\
8 & $(-1)^{\beta_i+1}$ & $x_i+ (-1)^{\beta_i+1}$ &$-(x_i+(-1)^{\beta_i+1})$ \\
9 & $(-1)^{\beta_i+1}$ & $x_i+ (-1)^{\beta_i+1}$ & $-(x_i+(-1)^{\beta_i+1})$\\
10 & $(-1)^{\beta_i+1}$ & $x_i+ (-1)^{\beta_i+1}$ & $-(x_i+(-1)^{\beta_i+1})$\\ \hline
11 & $(-1)^{\beta_i+1}$ & $x_i+ (-1)^{\beta_i+1}$ & $-x_i$\\
12 & $(-1)^{\beta_i}$ & $x_i+ (-1)^{\beta_i}$ & $-x_i$ \\
13 & $(-1)^{\beta_i}$ & $x_i+ (-1)^{\beta_i}$ & $-x_i$ \\ \hline
\end{tabular}
\end{center}
\end{figure}

\newpage

\begin{figure}[h] \label{thirteen matrices}
\caption{Thirteen matrices which arise in the proof of extremality when $k=2$}
\begin{tabular}{cc}
\begin{tabular}{c}
Case 1:\\
$\left( \begin{array}{ccccc}
1 & 2 & 2 & 1 & -2\\
0 & 1 & 1 & -1 & -1\\
0 & 0& 1& 2 & 0\\
a_{\tau} & b_{\tau} & 0& d_{\tau} & e_{\tau} \\
A_{\tau} & B_{\tau} & 0 & D_{\tau} & E_{\tau} \\
\end{array} \right).$
\end{tabular}
&
\begin{tabular}{c}
Case 2:\\
$\left( \begin{array}{ccccc}
1 & 2 & 1 & -1 & -1\\
0 & 1 & 2 & 0 & -1\\
0 & 0& 1& 2 & 0\\
a_{\tau} & b_{\tau} & 0& d_{\tau} & e_{\tau} \\
A_{\tau} & B_{\tau} & 0 & D_{\tau} & E_{\tau} \\
\end{array} \right).$
\end{tabular}
\\
\mbox{} & \\
\begin{tabular}{c}
Case 3:\\
$\left( \begin{array}{ccccc}
1 & 2 & 2 & 1 & -2\\
0 & 1 & 1 & -1 & -2\\
0 & 0& 1& 1 & -1\\
a_{\tau} & b_{\tau} & 0& d_{\tau} & e_{\tau} \\
A_{\tau} & B_{\tau} & 0 & D_{\tau} & E_{\tau} \\
\end{array} \right).$
\end{tabular}
&
\begin{tabular}{c}
Case 4:\\
$\left( \begin{array}{ccccc}
1 & 2 & 1 & -1 & -2\\
0 & 1 & 2 & 0 & -2\\
0 & 0& 1& 1 & -1\\
a_{\tau} & b_{\tau} & 0& d_{\tau} & e_{\tau} \\
A_{\tau} & B_{\tau} & 0 & D_{\tau} & E_{\tau} \\
\end{array} \right).$
\end{tabular}
\\
\mbox{} & \\
\begin{tabular}{c}
Case 5:\\
$\left( \begin{array}{ccccc}
1 & 2 & 1 & -1 & -4 \\
0 & 1 & 1 & -1 & -2 \\
0 & 0 & 1 & 1 & -1 \\
a_{\tau} & b_{\tau} & d_{\tau} & e_{\tau} & 2e_{\tau} \\
A_{\tau} & B_{\tau} & D_{\tau} & E_{\tau} & 2E_{\tau} 
\end{array} \right)$
\end{tabular}
&
\begin{tabular}{c}
Case 6:\\
$\left( \begin{array}{ccccc}
1 & 2 & 1 & -1 & -2 \\
0 & 1 & 2 & 0 & -2 \\
0 & 0 & 1 & 1 & -1 \\
a_{\tau} & b_{\tau} & d_{\tau} & e_{\tau} & 2e_{\tau} \\
A_{\tau} & B_{\tau} & D_{\tau} & E_{\tau} & 2E_{\tau} 
\end{array} \right)$
\end{tabular}
\\
\mbox{} & \\
\begin{tabular}{c}
Case 7:\\
$\left( \begin{array}{cccc}
1 & 2 & 1 & -1 \\
0 & 1 & 1 & -1 \\
A_{\tau} & B_{\tau} & D_{\tau} & E_{\tau} \\
a_{\tau} & b_{\tau} & d_{\tau} & e_{\tau}
\end{array} \right)$
\end{tabular}
&
\begin{tabular}{c}
Case 8:\\
$\left( \begin{array}{cccc}
1 & 2 & 1 & -1 \\
0 & 1 & 2 & 0 \\
A_{\tau} & B_{\tau} & D_{\tau} & E_{\tau} \\
a_{\tau} & b_{\tau} & d_{\tau} & e_{\tau}
\end{array} \right)$
\end{tabular}
\\
\mbox{} & \\
\begin{tabular}{c}
Case 9:\\
$\left( \begin{array}{cccccc}
1 & 2 & 1 &0  & 0 & -1 \\
0 & 1 & 2 & 1 & 0 & -1 \\
A_{\tau} & B_{\tau} & D_{\tau} & 2D_{\tau} & 2D_{\tau} & E_{\tau} \\
0 & 0 & 0 & 1 & 1 & -1 \\
0 & 0 & 0 & 0 & 1 & 1 \\
a_{\tau} & b_{\tau} & d_{\tau} & 2d_{\tau} &  2d_{\tau} & e_{\tau} \\
\end{array} \right)$
\end{tabular}
&
\small
\begin{tabular}{c}
Case 10, 13:\\
$\left( \begin{array}{cccccccccc}
1 & 2 & 1 &0  & \cdots & & &  & & -1 \\
0 & 1 & 2 & 1 & 0 & \cdots & & & & -1 \\
A_{\tau} & B_{\tau} & D_{\tau} & 2D_{\tau} & 2D_{\tau} & D_{\tau} & 0
& \cdots & 0 & E_{\tau} \\
0 & 0 & 0 & 1 & 1 & 0 & \cdots &  & &  -1 \\
0 & 0 & 0& 0 & 1 & 2 & 1 & 0 & \cdots & -1 \\
0 & 0 & 0 & 0 & 0& 1 & 2 & 1 & \ddots & -1 \\
0 & 0 & 0 & 0& 0 & \ddots & \ddots & \ddots & \ddots & \vdots \\
0 & 0 & 0& 0& 0 & 0 & 0 & 1 & 2 & 0 \\
0 & 0& 0 & 0& 0 & 0 & 0 & 0 & 1 & 1 \\
a_{\tau} & b_{\tau} & d_{\tau} & 2d_{\tau} &  2d_{\tau} & d_{\tau} & 0
& \cdots & 0 & e_{\tau} \\
\end{array} \right)$
\end{tabular}
\normalsize
\\

\mbox{} & \\

\begin{tabular}{c}
Case 11:\\
$\left( \begin{array}{ccccc}
1 & 2 & 1 & -1 & -2 \\
0 & 1 & 1 & 0 & -1 \\
A_{\tau} & B_{\tau} & D_{\tau} & 2D_{\tau} & E_{\tau} \\
0 & 0 & 0 & 1 & 0 \\
a_{\tau} & b_{\tau} & d_{\tau} & 2d_{\tau} & e_{\tau} \\
\end{array} \right)$
\end{tabular}

&

\begin{tabular}{c}
Case 12:\\
$\left( \begin{array}{ccccc}
1 & 2 & 1 & 0 & -1 \\
0 & 1 & 2 & 1 & -1 \\
A_{\tau} & B_{\tau} & D_{\tau} & 2D_{\tau} & E_{\tau} \\
0 & 0 & 0 & 1 & 0 \\
a_{\tau} & b_{\tau} & d_{\tau} & 2d_{\tau} & e_{\tau} \\
\end{array} \right)$
\end{tabular}

\end{tabular}

\end{figure}

\newpage

\newcounter{lastbib}

\setcounter{lastbib}{\value{bib}}

\section*{Software Packages Referenced}

\end{document}